\documentclass[ejs, preprint]{imsart}

\usepackage[english]{babel}
\usepackage[utf8]{inputenc}
\usepackage{amsthm,amsmath, amssymb}
\usepackage{graphicx}
\usepackage{xcolor}
\usepackage{enumitem}
\usepackage{float}
\RequirePackage[numbers]{natbib}
\RequirePackage[colorlinks,citecolor=blue,urlcolor=blue]{hyperref}

\startlocaldefs

\DeclareMathOperator*{\argmin}{arg\!\min}

\makeatletter
\newenvironment{proofof}[1]{\par
  \pushQED{\qed}%
  \normalfont \topsep6\p@\@plus6\p@\relax
  \trivlist
  \item[\hskip\labelsep
        \bfseries
    Proof of #1\@addpunct{.}]\ignorespaces
}{%
  \popQED\endtrivlist\@endpefalse
}
\makeatother

\theoremstyle{definition}
\newtheorem{defi}{Definition}[section]
  
\theoremstyle{remark}

\newtheorem{rem}[defi]{Remark}

\theoremstyle{plain}
\newtheorem{lem}[defi]{Lemma}
\newtheorem{cor}[defi]{Corollary}
\newtheorem{thm}[defi]{Theorem}
\newtheorem{prp}[defi]{Proposition}

\endlocaldefs

\begin{document}

\begin{frontmatter}

\title{
  Smoothed residual stopping for statistical inverse problems via truncated SVD
  estimation\thanksref{t1}
}
\runtitle{Smoothed residual stopping for statistical inverse problems}
\thankstext{t1}{I am very grateful for long discussions with Markus Reiß and
  Martin Wahl and the comments of an associate editor and two anonymous
  referees.}

\author{
  \fnms{Bernhard} \snm{Stankewitz}\corref{}\ead[label=e1]{stankebe@math.hu-berlin.de}
}
\address{
  Institut für Mathematik\\ 
  Humboldt-Universität zu Berlin, Germany\\
  \printead{e1}
}

\runauthor{B. Stankewitz}

\begin{abstract}
  This work examines under what circumstances adaptivity for truncated SVD
  estimation can be achieved by an early stopping rule based on the smoothed
  residuals \( \| ( A A^{\top} )^{\alpha / 2} ( Y - A \widehat{\mu}^{( m )})
  \|^{2} \). 
  Lower and upper bounds for the risk are derived, which show that moderate
  smoothing of the residuals can be used to adapt over classes of signals with
  varying smoothness, while oversmoothing yields suboptimal convergence rates.
  The range of smoothness classes for which adaptation is possible can
  be controlled via $ \alpha $. 
  The theoretical results are illustrated by Monte-Carlo simulations.
\end{abstract}




\end{frontmatter}

\numberwithin{equation}{section}
\setcounter{page}{+1}

\section{Introduction}
\label{sec_Introduction}

\subsection{Preliminaries on early stopping}
\label{ssec_PreliminariesOnEarlyStopping}

In machine learning and statistics, one of the central problems is that of
coping with the generalisation error or, put another way, choosing the correct
tuning parameter for an estimation procedure.
For iterative procedures, the generalisation error typically decreases up to a
point at which the algorithm begins to overfit.
Hence, the problem becomes that of choosing a suitable iteration step.
Classically, this problem would be addressed by 
model selection criteria such as \emph{cross-validation}, \emph{unbiased risk
estimation} or \emph{Lepski's balancing principle}.
These criteria, however, require that all estimators we want to choose from be
computed and then compared against each other.
For high dimensional problems in particular, this may come at a computationally
prohibitive cost.
An alternative are \emph{early stopping rules}, which halt the procedure at an
iteration $ \widehat{m} $ depending only on the iterates of index $ m \le
\widehat{m} $ and potentially additional quantities computed up to that point.
Since these require the computation of much fewer iterates, they present the
potential of simultaneously achieving computational and statistical
efficiency.

In order to locate this work in the literature on early stopping, we shortly
discuss three exemplary approaches:
In practical machine learning applications, early stopping rules are widely
adopted.
They are usually based on a well founded heuristic understanding of the
regularisation properties of early stopping.
For example, the user may split the data into training and validation sets and
iterate the learning algorithm on the training set until the validation error
does not improve any further, see Chapter 7 in \citet{GoodfellowEtal2016}. 
However, proper theoretical results for such rules are lacking. 

Some progress towards theoretical foundations of stopping rules has been made in
the kernel learning literature.
For the regression problem of learning $ f^{*} $ from data generated by $ Y =
f^{*}(X) + \varepsilon $, stopping rules have been suggested for gradient
descent procedures, initially, via oracle stopping times, which cannot be
computed from the data, see \citet{BuehlmannYu2003} and \citet{CaponettoEtal2007}.
Later, these have been converted to data dependent rules using empirical
versions of Gaussian and Rademacher complexities, see
\citet{RaskuttiEtal2014} and \citet{YangEtal2019}. 
For example, in \cite{RaskuttiEtal2014}, the authors learn $ f^{*} $ by
applying gradient descent to the problem $ \min_{z \in \mathbb{R}^{n}} \| Y -
\sqrt{K} z \|^{2} $, where $ Y $ is the vector of observations and $ K $ is the
empirical kernel matrix.
The procedure is stopped at
\begin{align}
  \widehat{T}: 
  = \inf \Big\{ 
      t \in \mathbb{N}: 
      \frac{1}{n} \sum_{i = 1}^{n}
      \min \{ \widehat{\lambda}_{i}, t^{- 1 / 2} \} 
      > ( \sigma t )^{-1} 
    \Big\} 
    - 1,
\end{align}
where the $ ( \widehat{\lambda}_{i} ) $ are the scaled eigenvalues of $ K $ and
$ \sigma $ is the noise level (up to a constant).
This rule is computable from the data and allows to adapt to the complexity of
the underlying kernel space.
Yet, other than the heuristic stopping rule above, this rule
structurally cannot adapt to the true data generating process. 
The kernel matrix and hence the sequence $ ( \widehat{\lambda}_{i} )_{i = 1,
\dots n} $ only depends on the design variables.
Therefore, $ \widehat{T} $ does not depend on $ f^{*} $ itself and will overfit
when the true smoothness of $ f^{*} $ is larger than the minimal smoothness of
functions from the kernel space.

Finally, additional progress has been made in the literature on statistical
inverse problems, which is another important framework for learning, see
e.g. \citet{RosascoEtal2005}.  
Blanchard, Hoffmann and Rei\ss\ \cite{BlanchardEtal2018a} consider early stopping
for a $ D $-dimensional discretisation of the inverse problem $ Y = A \mu +\delta \dot W $ with white noise $ \dot W $ and the sequence $ (
\widehat{\mu}^{( m )} )_{m = 1, \dots, D} $ of truncated SVD estimators.
They analyse the stopping rule 
\begin{align}
  \label{eq_TauForNonSmoothedResiduals}
  \tau:
  = \inf \{ 
      m \in \mathbb{N} \cup \{ 0 \}:
      \| Y - A \widehat{\mu}^{( m )} \|^{2} \le \delta^{2} D
    \}
\end{align}
based on the \emph{discrepancy principle}, which is well studied for
deterministic inverse problems, see e.g. \citet{EnglEtal1996}. 
This problem is similar to \cite{RaskuttiEtal2014} in that minimising $ \|
Y - \sqrt{K} z \|^{2} $ can also be understood as solving a finite dimensional
inverse problem.
The stopping rule $ \tau $, however, structurally differs from $ \widehat{T} $
in that, via $ Y $,  it takes the true signal $ \mu $ into account. 
Indeed, the authors prove that, up to a dimension dependent error term, stopping
according to $ \tau $ satisfies an oracle inequality, which yields rate optimal
adaptation simultaneously over a range of Sobolev-type ellipsoids of differing
smoothness.
Therefore, while the setting in \cite{BlanchardEtal2018a} is less general than
in the kernel literature, their version of early stopping is more comprehensive.
In addition, their setting can be understood as a prototypical model of an
iterative estimation procedure.

The analysis in this work is a continuation of the third approach above, 
where we stop using the $ ( \alpha ) $-smoothed residuals $
\| ( A A^{\top} )^{\alpha / 2} ( Y - A \widehat{\mu}^{( m )} ) \|^{2} $
for general $ \alpha > 0 $ instead.
In the next section, we motivate in detail why this should be considered and
what can be gained by it.


\subsection{Model and problem formulation}
\label{ssec_ModelAndProblemFormulation}

We recall in detail the setting in Blanchard, Hoffmann and Rei\ss\
\cite{BlanchardEtal2018a}:
They consider problems of the form
\begin{align}
  \label{eq_StatisticalInverseProblem}
  Y = A \mu + \delta \dot W,  
\end{align}
where \( A: H_{1} \to H_{2} \) is a linear bounded operator between real Hilbert
spaces, \( \mu \in H_{1} \) is the signal of interest, \( \delta > 0 \) is the
noise level and \( \dot W \) is a Gaussian white noise in \( H_{2} \). 
In any practical application, the problem has to be discretised by the user. 
Therefore, we can assume that \( H_{1} = \mathbb{R}^{D} \) and \( H_{2} =
\mathbb{R}^{P} \) for \( D \le P \), which both are possibly very large.
Further, assume that \( A: \mathbb{R}^{D} \to \mathbb{R}^{P} \) is
one-to-one. 
By transforming \eqref{eq_StatisticalInverseProblem}, using the singular value
decomposition (SVD) of \( A \), we arrive at the Gaussian vector observation
model
\begin{align}
  \label{eq_SVDRepresentation}
  Y_{i} = \lambda_{i} \mu_{i} + \delta \varepsilon_{i}, 
  \qquad i = 1, \dots, D.
\end{align}
\( \lambda_{1} \ge \lambda_{2}, \dots, \lambda_{D} > 0 \) are the singular
values of \( A \), \( ( \mu_{i} )_{i \le D} \) the coefficients of \( \mu \) in
the orthonormal basis of singular vectors and \( ( \varepsilon_{i} )_{i \le D}
\) are independent standard Gaussian random variables. 

In order to recover the signal \( \mu = ( \mu_{i} )_{i \le D} \) from the
observation of \eqref{eq_SVDRepresentation}, we use the \emph{truncated SVD}
(\emph{cut-off}) estimators \( \widehat{\mu}^{( m )}, m = 0, \dots, D \) given
by
\begin{align}
  \label{eq_TruncatedSVDEstimators}
  \widehat{\mu}_{i}^{( m )}: 
  = \mathbf{1} \{ i \le m \} \lambda_{i}^{-1} Y_{i}, 
  \qquad i = 1, \dots, D.
\end{align}
For a fixed index \( m \), the risk (expected squared Euclidean error) of \(
\widehat{\mu}^{( m)} \) can be decomposed into a bias and a variance term:
\begin{align}
  \label{eq_Bias}
  B^{2}_{m}(\mu): 
  & = \| \mathbb{E} \widehat{\mu}^{( m )} - \mu \|^{2}  
    = \sum_{i = m + 1}^{D} \mu_{i}^{2} \\ 
  \label{eq_Variance}
  \qquad \text{ and } \qquad 
  V_{m}: 
  & = \mathbb{E} \| \widehat{\mu}^{( m )} - \mathbb{E} \widehat{\mu}^{( m )} \|^{2} 
    = \sum_{i = 1}^{m} \lambda_{i}^{- 2} \delta^{2}. 
\end{align}
In particular, the estimators are ordered with decreasing bias and increasing
variance in \( m \). 
We reemphasise the importance of this setting as a prototypical model of an
iterative method.
Note that the truncated SVD-estimators are iterative in the sense that the
SVD of the operator has to be computed alongside the
estimators.
This is the case, since in practice, we cannot expect the observation vector $ Y
$ to be represented in an SVD basis, see also the detailed discussion in
\cite{BlanchardEtal2018a} and the references therein.
Other iterative methods often share important qualitative features with cut-off
estimation.
Therefore, results from this simple framework typically carry over
to more complex settings.
For example, Blanchard, Hoffmann and Rei\ss\ \cite{BlanchardEtal2018b} transfer
the results of \cite{BlanchardEtal2018a} to general regularisation schemes,
including gradient descent.  

In \cite{BlanchardEtal2018a}, the authors consider stopping according to the
discrepancy principle, i.e. at the smallest \( m \) which satisfies
\begin{align}
  \label{eq_ClassicalResidualStoppingCondition}
  \| Y - A \widehat{\mu}^{( m )} \|^{2} \le \kappa
\end{align}
for a suitable critical value \( \kappa > 0 \). 
Their analysis shows that generally, stopping according to the condition in
\eqref{eq_ClassicalResidualStoppingCondition} is optimal (in terms of an oracle
inequality) up to a dimension dependent error term, which stems from the
variability of the residuals.
For signals $ \mu $, which are not too smooth relative to the approximation
dimension $ D $, this term is of lower order. 
More precisely, \eqref{eq_ClassicalResidualStoppingCondition} yields
optimal results simultaneously for all signals satisfying \(
m^{\mathfrak{b}}(\mu) \gtrsim \sqrt{D} \), where
\begin{align}
  m^{\mathfrak{b}}(\mu): 
  = \inf \{ m \ge 0: B^{2}_{m}(\mu) \le V_{m} \} 
\end{align}
is the index at which balance between the squared bias and variance is obtained.
Otherwise, random deviations in the residuals systematically lead to stopping
times which are too large. 

Alternatively, \citet{BlanchardMathe2012} apply the discrepancy
principle to the normal equation \( A^{\top} Y = A^{\top} A \mu \) and stop
according to 
\begin{align}
  \label{eq_DiscrepancyPrinciple}
  \| A^{\top} ( Y - A \widehat{\mu}^{( m )} ) \|^{2} 
  = \| ( A A^{\top} )^{1 / 2} ( Y - A \widehat{\mu}^{( m )} ) \|^{2} 
  \le \kappa, 
\end{align}
i.e. the residuals are smoothed by \( ( A A^{\top} )^{1 / 2} \). 
This is motivated by the fact that in the infinite-dimensional problem, $ A^{*}
\dot W $ can be represented as an element of $ H_{1} $ when $ A $ is
Hilbert-Schmidt.
The condition in \eqref{eq_DiscrepancyPrinciple} is able to 
to control the stochastic part of the residuals and avoid the
dimension-dependency from \cite{BlanchardEtal2018a}.
Yet, it typically results in suboptimal convergence rates, since the variability
of the residuals is reduced too much, which leads to stopping times which are
too small.

These results raise the question of whether there is a stopping criterion in
between \eqref{eq_ClassicalResidualStoppingCondition} and
\eqref{eq_DiscrepancyPrinciple} which is able to mitigate the
dimension-dependency from \cite{BlanchardEtal2018a} and thereby increase the
range of signals for which adaptation is possible without slipping into the
suboptimal regime discussed in \cite{BlanchardMathe2012}.  
A very natural consideration is to smooth the residuals by a general power \(
\alpha \ge 0 \) of \( ( A A^{\top} )^{1 / 2} \) and stop at the smallest index
\( m \) which satisfies
\begin{align}
  \label{eq_SmoothedResidualStopping}
  R_{m, \alpha}^{2}: 
  = \| ( A A^{\top} )^{\alpha / 2} ( Y - A \widehat{\mu}^{( m )} ) \|^{2} 
  \le \kappa,
\end{align}
where \( R^{2}_{m, \alpha} \) are the (\( \alpha \)-)\emph{smoothed residuals}.
The main contribution of this paper is to answer the posed question in the
affirmative for the criterion in \eqref{eq_SmoothedResidualStopping}, provided
that the inverse problem is moderately ill-posed.
Smoothing with \( \alpha > 0 \) reduces the variability of \( R^{2}_{m, \alpha}
\), which mitigates the constraint from \cite{BlanchardEtal2018a}.
For values of \( \alpha \) which are small relative to the decay of the singular
values of \( A \), smoothing does not produce suboptimal rates.
Additionally, it is possible to eliminate the dimension constraint entirely
before the oversmoothing effect from \cite{BlanchardMathe2012} manifests. 
In order to further motivate stopping according to $ R_{m, \alpha}^{2} $, we
compare it to other possible generalisations of the discrepancy principle:

\begin{rem}[Other discrepancy-type rules]
  \label{rem_OtherDiscrepancyTypeRules}
  \ 
  \begin{enumerate}[label=(\alph*)]
    \item \citet{BlanchardMathe2012} also choose a stopping criterion in between
      \eqref{eq_ClassicalResidualStoppingCondition}  and
      \eqref{eq_DiscrepancyPrinciple} in order to guarantee optimality. 
      They weigh the residuals in \eqref{eq_DiscrepancyPrinciple} further by \(
      \varrho_{\lambda}(A^{\top} A) \) for \( \varrho_{\lambda}(t): = 1 /
      \sqrt{t + \lambda} \), \( t > 0 \) and a tuning parameter \( \lambda \).
      In their framework, however, the final choice of $ \lambda $ directly
      depends on the smoothness of the true signal and only yields optimal results
      for this smoothness class.
      Therefore, their stopping criterion will not adapt simultaneously to
      signals of varying smoothness, which is precisely the goal of our
      analysis.    

    \item Other well founded variations of the discrepancy principle mostly take
      the form 
      \begin{align}
        \label{eq_HmResiduals}
        \| H_{m}(A A^{\top}) ( Y - A \widehat{\mu}^{( m )} ) \|^{2} 
        \le \kappa,
      \end{align}
      i.e. the weight of $ ( A A^{\top} ) $ depends on $ m $, see e.g.
      \citet{EnglEtal1996}. 
      Compared to the smoothed residuals, such a rule is computationally more
      expensive:
      In our setting, the computation of the first $ m $ estimators roughly
      requires $ O(m D^{2}) $ operations, see \cite{BlanchardEtal2018a}. 
      With the update $ R_{m + 1, \alpha}^{2} = R _{m, \alpha}^{2} -
      \lambda_{m + 1}^{2 \alpha} Y_{m + 1}^{2} $, the additional computational cost
      of the smoothed residuals is negligible.
      Note that the $ m $-th eigenvalue $ \lambda_{m} $ already has to be computed
      for $ \widehat{\mu}^{( m )} $.
      In contrast, computing \eqref{eq_HmResiduals} for $ i = 0, \dots, m $
      potentially requires $ O(m D^{2}) $ operations itself. 
      If we regard early stopping as a tool to treat the computational
      complexity of the problem, this provides further motivation for the
      $ ( \alpha ) $-smoothed residuals.
  \end{enumerate}
\end{rem}

The remainder of the paper is structured as follows:
In Section \ref{sec_FrameworkForTheAnalysisAndMainResults}, we collect the
structural assumptions of the analysis and provide an interpretation of the
smoothed residual stopping procedure in \eqref{eq_SmoothedResidualStopping}
as estimating the bias of a smoothed version of the risk.
At the end, we present the main results of the paper, which are
derived in Section \ref{sec_DerivationOfTheMainResults}. 
Its constraints in terms of lower bounds are explored in Section
\ref{sec_ConstraintsInTermsOfLowerBounds}. 
Finally, Section \ref{sec_DiscussionAndSimulations} discusses different choices
for the smoothing parameter \( \alpha \) and illustrates the results by
Monte-Carlo simulations.



\section{Framework for the analysis and main results}
\label{sec_FrameworkForTheAnalysisAndMainResults}

\subsection{Structural assumptions}
\label{ssec_StructuralAssumptions}

Throughout the paper, we assume that the inverse problem is moderately
ill-posed, i.e. the singular values \( ( \lambda_{i} )_{i \le D} \) satisfy a
\emph{polynomial spectral decay} assumption of the form
\begin{align}
  \label{eq_PSD}
  C_{A}^{-1} i^{- p} \le \lambda_{i} \le C_{A} i^{- p}, 
  \qquad i = 1, \dots, D
  \qquad \qquad ( \text{PSD}(p, C_{A}) ) 
\end{align}
for some \( p \ge 0 \) and \( C_{A} \ge 1 \).  
By dividing Equation \eqref{eq_SVDRepresentation} by \( \lambda_{1} \), we can
further assume that \( \lambda_{i} \le 1 \), \( i = 1, \dots, D \). 
Additionally, we always require that the critical value \( \kappa \) satisfies
\begin{align}
  \label{eq_AssumptionOnKappa}
  | \kappa - \sum_{i = 1}^{D} \lambda_{i}^{2 \alpha} \delta^{2} | 
  \le C_{\kappa} s_{D} \delta^{2} 
  \qquad \text{ with } \qquad 
  s_{D}^{2}: = 2 \sum_{i = 1}^{D} \lambda_{i}^{4 \alpha } 
\end{align}
for an absolute constant \( C_{\kappa} > 0 \).  

Note that $ \sum_{i = 1}^{D} \lambda_{i}^{2} \delta^{2} $ is the
expectation of the smoothed residuals for the zero signal at $ m = 0 $, since 
\begin{align}
  R^{2}_{0, \alpha} 
  = \sum_{i = 1}^{D} \Big( 
      \lambda_{i}^{2 + 2 \alpha} \mu_{i}^{2}  
    + 2 \lambda_{i}^{1 + 2 \alpha} \mu_{i} \delta \varepsilon_{i}   
    + \lambda_{i}^{2 \alpha} \delta^{2} \varepsilon_{i}^{2}
    \Big). 
\end{align}
Similarly, $ s_{D} \delta^{2} $ is the standard deviation of the 
dominant stochastic part of the term above.
Therefore, \eqref{eq_AssumptionOnKappa} states that up to small deviations, $
\kappa $ should be chosen as the expectation of the smoothed residuals in the
pure noise case.

In the following, we denote essential inequalities up to an absolute constant by
``\( \lesssim, \gtrsim, \sim \)".
Further dependencies on \( \alpha \), the operator \( A \), i.e. \( p \) and \(
C_{A} \), and \( C_{\kappa} \), are denoted by indices \( \alpha, A \) and \(
\kappa \).  
Finally, we assume that all smoothing indices \( \alpha \) are bounded from
above by some \( \bar \alpha > 0 \).
This guarantees that \( \lambda_{i}^{\alpha} \sim_{A} i^{- \alpha p} \), \(
i \le D \).  
Under \hyperref[eq_PSD]{\( ( \text{PSD}(p, C_{A}) ) \)}, the order of \( s_{D}
\) is given by
\begin{align}
  \label{eq_SizeOfSD}
  s_{D} \sim_{\alpha, A} \begin{cases}
                           D^{1 / 2 - 2 \alpha p}, & \alpha p < 1 / 4, \\ 
                           \log D,                 & \alpha p = 1 / 4, \\ 
                           1,                      & \alpha p > 1 / 4. 
                         \end{cases}
\end{align}
The fact that the order of \( s_{D} \) is decreasing in \( \alpha \) will later
allow to relax the constraint from \citet{BlanchardEtal2018a}. 
The variance of \( \widehat{\mu}^{( m )} \) is of order
\begin{align}
  \label{eq_SizeOfTheVariance}
  V_{m} 
  = \sum_{i = 1}^{m} \lambda_{i}^{- 2} \delta^{2} 
  \sim_{A} m^{2 p + 1} \delta^{2}. 
\end{align}
For the analysis of lower bounds in Section
\ref{sec_ConstraintsInTermsOfLowerBounds}, we consider signals from
\emph{Sobolev-type ellipsoids} 
\begin{align}
  \label{eq_SobolevEllipsoid}
  H^{\beta}(r, D): 
  = \Big\{ 
      \mu \in \mathbb{R}^{D}: 
      \sum_{i = 1}^{D} i^{2 \beta} \mu_{i}^{2} \le r^{2}
    \Big\} 
  \qquad \text{for some } \beta \ge 0, r > 0. 
\end{align}
For \( \mu \in H^{\beta}(r, D) \), we have the upper bound
\begin{align}
  \label{eq_SizeOfTheBias}
  B^{2}_{m}(\mu) 
  = \sum_{i = m + 1}^{D} \mu_{i}^{2} 
  \le ( m + 1 )^{- 2 \beta} r^{2} 
\end{align}
for the squared bias of \( \widehat{\mu}^{( m )} \). 
The bounds in \eqref{eq_SizeOfTheVariance} and \eqref{eq_SizeOfTheBias} are
balanced at the order of the \emph{minimax-truncation index} 
\begin{align}
  \label{eq_MinimaxTruncationIndex}
  t^{mm}_{\beta, p, r} 
  = t^{mm}_{\beta, p, r}(\delta): 
  = ( r^{2} \delta^{- 2} )^{1 / ( 2 \beta + 2 p + 1 )}. 
\end{align}

Taking the asymptotic view that $ D = D(\delta) \to \infty $ for $ \delta \to 0
$, the rate $ v_{\delta}^{2} $ is optimal in the minimax sense if there exist
estimators $ ( \widehat{\mu}_{\delta})_{\delta > 0} $ in the models
corresponding to the ellipsoids $ H^{\beta}(r, D(\delta)) $ such that  
\begin{align}
  \limsup_{\delta \to 0} 
  v_{\delta}^{- 2} 
  \sup_{\mu \in H^{\beta}(r, D(\delta))} 
  \mathbb{E} \| \widehat{\mu}_{\delta} - \mu \|^{2} 
  < \infty 
\end{align}
and 
\begin{align}
  \liminf_{\delta \to 0} 
  v_{\delta}^{- 2} \inf_{\widehat{\mu}}
  \sup_{\mu \in H^{\beta}(r, D(\delta))} 
  \mathbb{E} \| \widehat{\mu} - \mu \|^{2} 
  > 0, 
\end{align}
where the infimum is taken over all estimators $ \widehat{\mu} $. 
A deterministic stopping index of the order of the minimax truncation index $
t^{mm}_{\beta, p, r} $ in \eqref{eq_MinimaxTruncationIndex} yields the rate
\begin{align}
  \label{eq_MinimaxRate}
  \mathcal{R}^{*}_{\beta, p, r}(\delta): 
  = r^{2} ( r^{- 2} \delta^{2})^{2 \beta / ( 2 \beta + 2 p + 1 )}.
\end{align}
This is the minimax rate in the infinite-dimensional Gaussian sequence model.
Note that lower bounding the minimax risk in the infinite-dimensional
case, up to a constant, only requires to consider alternatives in the first $
t^{mm}_{\beta, p, r} $ components, see e.g. Proposition 4.23 in
\citet{Johnstone2017}.
Therefore, if $ D(\delta) $ is chosen at least of the order of $ t^{mm}_{\beta,
p, r} $, the rate $ \mathcal{R}^{*}_{\beta, p, r}(\delta) $ is also minimax in
our setting.
In the asymptotic considerations, we will always assume that this is the case,
since we can also think of $ t^{mm}_{\beta, p, r} $ as the minimally sufficient
approximation dimension.
Indeed, the error of approximating a signal from an infinite-dimensional Sobolev
ellipsoid of smoothness $ \beta $ by a signal from $ H^{\beta}(r, D) $ will
only be negligible if $ D(\delta) \gtrsim t^{mm}_{\beta, p, r} $. 


\subsection{Smoothed residual stopping as bias estimation}
\label{ssec_SmoothedResidualStoppingAsBiasEstimation}

For a clearer formulation of the results, we introduce continuous versions of
the bias and the variance by linearly interpolating Equations 
\eqref{eq_Bias} and \eqref{eq_Variance}.
For \( t \in [ 0, D ] \), we set 
\begin{align}
  B^{2}_{t}(\mu): 
  & = ( \lceil t \rceil - t) \mu_{\lceil t \rceil}^{2} 
    + \sum_{i = \lceil t \rceil + 1}^{D} \mu_{i}^{2} \\ 
  \qquad \text{ and } \qquad 
  V_{t}: 
  & = \sum_{i = 1}^{\lfloor t \rfloor} \lambda_{i}^{- 2} \delta^{2} 
    + ( t - \lfloor t \rfloor ) \lambda_{\lceil t \rceil}^{- 2} \delta^{2}, 
\end{align}
where \( \lfloor t \rfloor \) and \( \lceil t \rceil \) are the floor and
ceiling functions, respectively.
We can define a continuous cut-off estimator \( \widehat{\mu}^{( t )} \) such
that \( \mathbb{E} \| \widehat{\mu}^{( t )} - \mu \|^{2} = B^{2}_{t}(\mu) +
V_{t}, t \in [ 0, D ] \): By randomising between the discrete estimators with
index \( \lfloor t \rfloor \) and \( \lceil t \rceil \), we set
\begin{align}
  \label{eq_ContinuousTruncatedSVDEstimator}
  \widehat{\mu}^{( t )}_{i}: 
  = ( 
      \mathbf{1} \{ i \le \lfloor t \rfloor \} 
    + \xi_{t} \mathbf{1} \{ i = \lceil t \rceil \}
    ) 
    \lambda_{i}^{-1} Y_{i}, 
  \qquad i = 1, \dots, D,
\end{align}
where \( \xi_{t} \) are Bernoulli random variables with success probabilities \(
t - \lfloor t \rfloor \) independent of everything else.
This also gives a continuous version of the smoothed residuals:
\begin{align}
  \label{eq:ContinuousSmoothedResiduals}
  R^{2}_{t, \alpha}: 
  & = \| ( A A^{\top} )^{\alpha / 2} ( Y - A \widehat{\mu}^{( t )} ) \|^{2} \\ 
  \notag 
  & = ( \mathbf{1} \{ t \ne \lceil t \rceil \} - \xi_{t} ) 
      \lambda_{\lceil t \rceil}^{2 \alpha} Y_{\lceil t \rceil}^{2} 
    + \sum_{i = \lceil t \rceil + 1}^{D} \lambda_{i}^{2 \alpha} Y_{i}^{2}
\end{align}
for \( t \in [ 0, D ] \).
The (\( \alpha \)-)\emph{smoothed residual stopping time} 
\begin{align}
  \label{eq_StoppingTime}
  \tau_{\alpha}: 
  = \inf \{ m \in \mathbb{N} \cup \{ 0 \}: R^{2}_{m, \alpha} \le \kappa \}
\end{align}
yet remains integer.
In the following, integer indices are denoted by \( m \) and continuous indices
are denoted by \( t \). 

Applying optional stopping to the martingale \( M_{m}: = \sum_{i = 1}^{m}
\lambda_{i}^{- 2} ( \varepsilon_{i}^{2} - 1 ) \), \( m \le D \), yields 
\begin{align}
  \label{eq_OptionalStoppingIdentity}
  \mathbb{E} \| \widehat{\mu}^{( \tau_{\alpha} )} - \mu \|^{2} 
  = \mathbb{E} \Big( 
      \sum_{i = \tau_{\alpha} + 1}^{D} \mu_{i}^{2} 
    + \sum_{i = 1}^{\tau_{\alpha}}
      \lambda_{i}^{- 2} \delta^{2} \varepsilon_{i} ^{2} 
    \Big) 
  = \mathbb{E} \big( B^{2}_{\tau_{\alpha}}(\mu) + V_{\tau_{\alpha}} \big).
\end{align}
Therefore, at best, the risk at \( \tau_{\alpha} \) behaves like the risk at
the \emph{classical oracle index} 
\begin{align}
  \label{eq_ClassicalOracle}
  t^{\mathfrak{c}} = t^{\mathfrak{c}}(\mu): 
  = \argmin_{t \in [ 0, D ]} 
    \mathbb{E} \| \widehat{\mu}^{( t )} - \mu \|^{2}
\end{align}
which minimises the risk over all deterministic stopping indices.
There is, however, no direct connection between \( \tau_{\alpha} \) and \(
t^{\mathfrak{c}} \). 
This is intrinsic to the sequential nature of the analysis, since at truncation
index \( t \), we cannot say anything about the behaviour of the bias for larger
indices.

For our purposes, we instead consider the \emph{balanced oracle index} 
\begin{align}
  \label{eq_BalancedOracle}
  t^{\mathfrak{b}} = t^{\mathfrak{b}}(\mu): 
  = \inf \{ t \ge 0 : B^{2}_{t}(\mu) \le V_{t} \}.
\end{align}
Due to the continuity of the functions \( t \mapsto V_{t} \) and \( t \mapsto
B^{2}_{t}(\mu) \), we have that at \( t^{\mathfrak{b}} \), squared bias and
variance balance exactly, i.e. \( B^{2}_{t^{\mathfrak{b}}}(\mu) =
V_{t^{\mathfrak{b}}} \).  
Furthermore, the balanced oracle risk is comparable to the classical oracle
risk:
The monotonicity of \( t \mapsto V_{t} \) and \( t \mapsto B^{2}_{t}(\mu) \)
yields
\begin{align}
  \label{eq_ComparabilityOfOracleRisks}
  \mathbb{E} \| \widehat{\mu}^{( t^{\mathfrak{b}} )} - \mu \|^{2} 
  =  B^{2}_{t^{b}}(\mu) + V_{t^{\mathfrak{b}}}
  \le 2 \mathbb{E} \| \widehat{\mu}^{( t^{\mathfrak{c}} )} - \mu \|^{2}
\end{align}
by distinguishing the cases \( t^{\mathfrak{c}} \le t^{\mathfrak{b}} \) and \(
t^{\mathfrak{c}} > t^{\mathfrak{b}} \). 
Assuming that the operator $ A $ and the noise level $ \delta $ are known,
knowledge of the bias is therefore enough to stop at an index at which the risk
is of the order of the classical oracle risk.

The smoothed residuals $ R^{2}_{t, \alpha} $ contain some information about the
bias:
We can write
\begin{align}
  \label{eq_InformationInTheSmoothedResiduals}
  \mathbb{E} R^{2}_{t, \alpha} 
  = B^{2}_{t, \alpha}(\mu) 
  + \sum_{i = 1}^{D} \lambda_{i}^{2 \alpha} \delta^{2} - V_{t, \alpha}, 
  \qquad t \in [ 0, D ],
\end{align}
where the \( \alpha \)-bias and the \( \alpha \)-variance
\begin{align}
  \label{eq_AlphaBias}
  B^{2}_{t, \alpha}(\mu): 
  & = ( \lceil t \rceil - t )
      \lambda_{\lceil t \rceil}^{2 + 2 \alpha} \mu_{\lceil t \rceil}^{2} 
    + \sum_{i = \lceil t \rceil + 1}^{D} 
      \lambda_{i}^{2 + 2 \alpha} \mu_{i}^{2} \\ 
  \label{eq_AlphaVariance}
  \qquad \text{ and } \qquad 
  V_{t, \alpha}: 
  & = \sum_{i = 1}^{\lfloor t \rfloor} \lambda_{i}^{2 \alpha} \delta^{2} 
    + ( t - \lfloor t \rfloor ) \lambda_{\lceil t \rceil}^{2 \alpha} \delta^{2} 
\end{align}
are smoothed versions of \( B^{2}_{t}(\mu) \) and \( V_{t} \). 
Since $ \lambda_{i} \le 1 $ for all $ i = 1, \dots, D $, the
smoothed quantities $ B^{2}_{t, \alpha} $ and $ V_{t, \alpha} $ are always
smaller than their nonsmoothed counterparts.
Analogously to $ t^{\mathfrak{b}} $, we define the $ \alpha $-\emph{balanced
oracle} 
\begin{align}
  \label{eq_AlphaBalancedOracle}
  t^{\mathfrak{b}}_{\alpha} = t^{\mathfrak{b}}_{\alpha}(\mu): 
  = \inf \{ t \ge 0: B^{2}_{t, \alpha}(\mu) \le V_{t, \alpha} \} 
\end{align}
at which the squared \( \alpha \)-bias and the $ \alpha $-variance balance.

\begin{figure}[t]
  \centering
  \includegraphics[width=\textwidth]{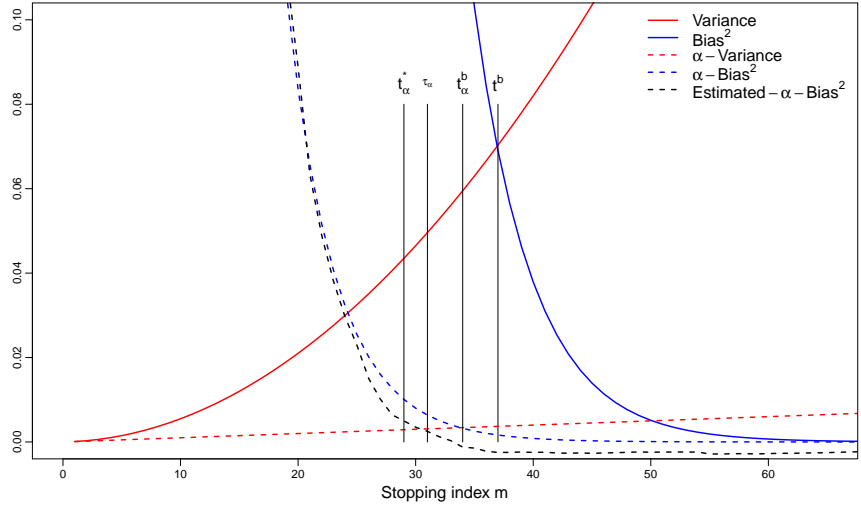}
  \caption{Bias estimation with oracle indices. Here, $ \alpha = 0 $ to ensure
  that all curves fit into one plot.}
  \label{fig_BiasEstimationWithOracleIndices}
\end{figure}

The stopping condition $ R^{2}_{m, \alpha} \le \kappa $ can be reformulated as
\begin{align}
  \widehat{B}^{2}_{m, \alpha}(\mu): 
  = R^{2}_{m, \alpha} + V_{m, \alpha} - \kappa 
  \le V_{m, \alpha},
\end{align}
which yields 
\begin{align}
  \tau_{\alpha} 
  = \inf \{ m \ge 0 : \widehat{B}^{2}_{m, \alpha}(\mu) \le V_{m, \alpha} \}.
\end{align}
Due to \eqref{eq_InformationInTheSmoothedResiduals}, \( \widehat{B}^{2}_{m,
\alpha}(\mu) \) is an unbiased estimator of \( B^{2}_{m, \alpha}(\mu) \) for \(
\kappa = \sum_{i = 1}^{D} \lambda_{i}^{2 \alpha} \delta^{2} \). 
Therefore, stopping according to $ \tau_{\alpha} $ can be understood as
estimating the $ \alpha $-bias and stopping when the estimate is smaller than
the $ \alpha $-variance.
For the specific choice of $ \kappa $ above, $ \tau_{\alpha} $ directly mimics $
t^{\mathfrak{b}}_{\alpha} $. 
For other choices of $ \kappa $, $ \tau_{\alpha} $ mimics the (\( \alpha
\)-)\emph{oracle-proxy index}
\begin{align}
  \label{eq_AlphaOracleProxy}
  t^{*}_{\alpha} 
  = t^{*}_{\alpha}(\mu): 
  = \inf \{ t \ge 0: \mathbb{E} \widehat{B}^{2}_{t, \alpha} \le V_{t, \alpha} \}  
  = \inf \{ t \ge 0: \mathbb{E} R^{2}_{t, \alpha} \le \kappa \}. 
\end{align}
This is illustrated in Figure \ref{fig_BiasEstimationWithOracleIndices}. 
The oracle-proxy index satisfies
\begin{align}
  \label{eq_ComparisonTbaTja}
  \begin{cases}
    t^{*}_{\alpha} > t^{\mathfrak{b}}_{\alpha}, & \kappa < \sum_{i = 1}^{D} \lambda_{i}^{2 \alpha} \delta^{2}, \\[3pt]
    t^{*}_{\alpha} = t^{\mathfrak{b}}_{\alpha}, & \kappa = \sum_{i = 1}^{D} \lambda_{i}^{2 \alpha} \delta^{2}, \\[3pt]
    t^{*}_{\alpha} < t^{\mathfrak{b}}_{\alpha}, & \kappa > \sum_{i = 1}^{D} \lambda_{i}^{2 \alpha} \delta^{2}.
  \end{cases}
\end{align}
Assumption \eqref{eq_AssumptionOnKappa} can therefore be
understood as a requirement on the difference between $ t^{*}_{\alpha} $ and $
t^{\mathfrak{b}}_{\alpha} $.
So far, this yields the following picture:
Approximately, \( \tau_{\alpha} \) is centred around the oracle
proxy \( t^{*}_{\alpha} \), which is close to the \( \alpha \)-balanced oracle
\( t^{\mathfrak{b}}_{\alpha} \) for an appropriate choice of \( \kappa \).
In turn, \( t^{\mathfrak{b}}_{\alpha} \) is related to the balanced oracle \(
t^{\mathfrak{b}} \) due to the connection between the bias and the variance and
their smoothed counterparts. 
Generally, we can therefore hope for adaptation as long as \(
t^{\mathfrak{b}}_{\alpha} \) and \( t^{\mathfrak{b}} \) are of the same size.

With respect to the difference between \( t^{\mathfrak{b}}_{\alpha} \) and \(
t^{\mathfrak{b}} \), we note:

\begin{lem}
  \label{lem_MonotonicityOfTheAlphaBalancedOracle}
  The mapping $ \alpha \mapsto t^{\mathfrak{b}}_{\alpha}, \alpha \ge 0  $ is
  monotonously decreasing in $ \alpha $.
  Further, $ t^{\mathfrak{b}}_{\alpha} \le t^{\mathfrak{b}} $ for all $ \alpha
  \ge 0 $.
\end{lem}

\begin{proof}
  Let $ \alpha, \alpha' \ge 0 $ with $ \alpha \le \alpha' $.
  Then, for any $ t \in [ 0, D ] $ which satisfies $ B^{2}_{t, \alpha}(\mu) \le
  V_{t, \alpha}  $, we have
  \begin{align}
    B^{2}_{t, \alpha'}(\mu) 
    \le \lambda_{\lceil t \rceil}^{2 ( \alpha' - \alpha )} B^{2}_{t, \alpha}(\mu) 
    \le \lambda_{\lceil t \rceil}^{2 ( \alpha' - \alpha )} V_{t, \alpha} 
    \le V_{t, \alpha'}. 
  \end{align}
  Analogous reasoning yields \( t^{\mathfrak{b}}_{\alpha} \le t^{\mathfrak{b}}
  \) for all \( \alpha \ge 0 \). 
\end{proof}

\noindent Therefore, smoothing increases the difference between $
t^{\mathfrak{b}}_{\alpha} $ and $ t^{\mathfrak{b}} $ and will generally induce
smaller stopping times \( \tau_{\alpha} \). 

Under \hyperref[eq_PSD]{\( ( \text{PSD}(p, C_{A}) ) \)}, we also have essential
upper bounds for $ t^{\mathfrak{b}}_{\alpha} $ and $ t^{\mathfrak{b}}_{\alpha}
$:
For $ t^{\mathfrak{b}} $, the bounds on the size of the bias and the variance in
\eqref{eq_SizeOfTheBias} and \eqref{eq_SizeOfTheVariance} show that
\begin{align}
  \label{eq_UpperBoundForTheBalancedOracle}
  t^{\mathfrak{b}}(\mu) 
  \lesssim_{A} t^{mm}_{\beta, p, r}(\delta) 
  = ( r^{2} \delta^{- 2} )^{1 / ( 2 \beta + 2 p + 1 )} 
  \qquad \text{for all } \mu \in H^{\beta}(r, D). 
\end{align}
For $ t^{\mathfrak{b}}_{\alpha} $, analogously to \eqref{eq_SizeOfTheBias} and
\eqref{eq_SizeOfTheVariance}, we obtain
\begin{align}
  \label{eq_SizeOfTheAlphaBias}
  B^{2}_{m, \alpha}(\mu) 
  & \lesssim_{A} r^{2} m^{- ( 2 \beta + 2 p + 2 \alpha p )} 
  \qquad \text{ for all } \mu \in H^{\beta}(r, D) \\
  \label{eq_SizeOfTheAlphaVariance}
  \quad \text{ and } \quad 
  V_{m, \alpha} 
  & \sim_{A} 
    \begin{cases}
      m^{1 - 2 \alpha p} \delta^{2} / ( 1 - 2 \alpha p ), & \alpha p < 1 / 2, \\ 
      \log(m) \delta^{2},                                 & \alpha p = 1 / 2, \\ 
      \delta^{2},                                         & \alpha p > 1 / 2
    \end{cases}
\end{align}
for sufficiently large values of \( m \ge 0 \).
Given that \( t^{\mathfrak{b}}_{\alpha} \) is large enough, this gives 
the essential upper bound
\begin{align}
  \label{eq_UpperBoundForTheAlphaBalancedOracle}
  t^{\mathfrak{b}}_{\alpha}(\mu) 
  \lesssim_{A} t^{mm}_{\beta, p, r, \alpha}(\delta) 
  \qquad \text{ for all } \mu \in H^{\beta}(r, D),
\end{align}
where
\begin{align}
  \label{eq_AlphaMinimaxTruncationIndex}
  t^{mm}_{\beta, p, r, \alpha} 
  = t^{mm}_{\beta, p, r, \alpha}(\delta): 
  = \begin{cases}
      ( ( 1 - 2 \alpha p ) r^{2} \delta^{- 2} )^{1 / ( 2 \beta + 2 p + 1 )},         & \alpha p < 1 / 2, \\ 
      ( r^{2} \delta^{- 2} / \log(r^{2} \delta^{- 2}) )^{1 / ( 2 \beta + 2 p + 1 )}, & \alpha p = 1 / 2, \\
      ( r^{2} \delta^{- 2} )^{1 / ( 2 \beta + 2 p + 2 \alpha p )},                   & \alpha p > 1 / 2 
    \end{cases}
\end{align}
is the \( \alpha \)-\emph{minimax truncation} index.

For \( \alpha p < 1 / 2 \), \( t^{mm}_{\beta, p, r, \alpha} \) is of the same
order as \( t^{mm}_{\beta, p, r} \), but smoothing shrinks \( t^{mm}_{\beta, p,
r, \alpha} \) by a power of \( ( 1 - 2 \alpha p ) \). 
In the same way, we obtain that for \( \alpha p \ge 1 / 2 \), the \( \alpha
\)-balanced oracle is of order strictly smaller than the minimax-truncation
index \( t^{mm}_{\beta, p, r}(\delta) \). 
Since there are signals \( \mu \in H^{\beta}(r, D) \), for which \(
t^{\mathfrak{b}}(\mu) \sim t^{mm}_{\beta, p, r}(\delta) \), we can therefore
only expect to achieve adaptation on \( H^{\beta}(r, D) \) as long as \( \alpha
p < 1 / 2 \). 


\subsection{Main results}
\label{ssec_MainResults}

Based on the understanding of the stopping procedure developed in Sections
\ref{ssec_StructuralAssumptions} and
\ref{ssec_SmoothedResidualStoppingAsBiasEstimation}, we can now formulate our
main theorem.  It provides an oracle inequality for the risk at $ \tau_{\alpha}
$ in terms of the risk at the balanced oracle $ t^{\mathfrak{b}} $. 

\begin{thm}[Balanced oracle inequality]
  \label{thm_BalancedOracleInequality}
  Assume \hyperref[eq_PSD]{\( ( \text{PSD}(p, C_{A}) ) \)} with \( \alpha p < 1
  / 2\) and \eqref{eq_AssumptionOnKappa}. 
  Then, there exists a constant $ C_{\alpha, A, \kappa} $ depending on $ \alpha,
  p, C_{A} $ and $ C_{k} $ such that 
  \begin{align*}
    \mathbb{E} \| \widehat{\mu}^{( \tau_{\alpha} )} - \mu \|^{2} 
    \le C_{\alpha, A, \kappa} \big( 
          \mathbb{E} \| \widehat{\mu}^{( t^{\mathfrak{b}} )} - \mu \|^{2} 
        + s_{D}^{( 2 p + 1 ) / ( 1 - 2 \alpha p )} \delta^{2} 
        \big).  
  \end{align*}
  For \( t^{\mathfrak{b}} \gtrsim_{\alpha, A, \kappa} s_{D}^{1 / ( 1 - 2 \alpha
  p )} \), the risk of stopping at \( \tau_{\alpha} \) is of the order of the
  balanced-oracle risk.
\end{thm}

\noindent Theorem \ref{thm_BalancedOracleInequality} is derived in Section
\ref{sec_DerivationOfTheMainResults}. 

We comment on the result:
$ s_{D}^{( 2 p + 1 ) / ( 1 - 2 \alpha p )} \delta^{2} $ is a dimension-dependent
error term.
Since \( V_{t} \sim_{A} t^{2 p + 1} \delta^{2} \), it is of order \( V_{s_{D}^{1
/ ( 1 - 2 \alpha p )}} \).
Its existence stems from the stochastic variability of the residuals,
which is discussed in Section \ref{ssec_UndersmoothingForAlphaPLe14}. 
Since the risk at $ t^{\mathfrak{b}} $ is of the order of $ V_{t^{\mathfrak{b}}}
$, this error term is of lower order as long as
\begin{align}
  \label{eq_AbstractRangeOfAdaptation}
  t^{\mathfrak{b}} 
  \gtrsim_{\alpha, A, \kappa} s_{D}^{1 / ( 1 - 2 \alpha p )} 
  \sim_{\alpha, A, \kappa} 
  \begin{cases}
    D^{\frac{1 / 2 - 2 \alpha p}{1 - 2 \alpha p}}, & \alpha p < 1 / 4, \\ 
    ( \log D )^{2},                                & \alpha p = 1 / 4, \\ 
    1,                                             & \alpha p > 1 / 4.
  \end{cases}
\end{align}
Equation \eqref{eq_AbstractRangeOfAdaptation} determines for what signals we
can obtain optimal estimation results and shows the advantage of smoothing:
For $ \alpha = 0 $, we obtain the same result as in \citet{BlanchardEtal2018a},
i.e. we need to require $ t^{\mathfrak{b}} \gtrsim_{A, \kappa} \sqrt{D} $. 
For values $ \alpha > 0 $, this constraint is weakened and thereby
guarantees that the dimension dependent error term is of lower order for a
larger class of signals.
For $ \alpha p = 1 / 4 $, the error is only a log-term.
For $ \alpha p > 1 / 4 $, it is of constant size.

Intuitively, under our assumptions, $ \tau_{\alpha} $ behaves like $
t^{\mathfrak{b}}_{\alpha} $. 
As seen in Lemma \ref{lem_MonotonicityOfTheAlphaBalancedOracle}, $
t^{\mathfrak{b}}_{\alpha} $ is monotonously decreasing in $ \alpha $. 
While decreasing the variance, 
smoothing therefore increases the squared bias
\(B^{2}_{t^{\mathfrak{b}}_{\alpha}}(\mu) \). 
For $ \alpha p < 1 / 2 $, this results in an increase in the constant $
C_{\alpha, A, \kappa} $. 
For $ \alpha p \ge 1 / 2 $, $ t^{\mathfrak{b}}_{\alpha} $ and $ t^{\mathfrak{b}}
$ can be of different order such that the squared bias at $
t^{\mathfrak{b}}_{\alpha} $ is strictly larger than the risk at $
t^{\mathfrak{b}} $. 
Then, an oracle inequality is no longer possible. 
The details of this are further discussed in Section
\ref{ssec_OversmoothingForAlphaPGe12}. 
One of the basic assumptions in \citet{BlanchardMathe2012} is that \( A \)
is Hilbert-Schmidt.
In our setting, this is the case when \( p > 1 / 2 \), which is the exact point
when the discrepancy principle for the normal equation, i.e. \( \alpha = 1 \),
loses the optimal rate.
Therefore, the above reasoning provides a nice explanation for their
nonoptimality result.

Finally, our result directly translates to an asymptotic minimax upper bound
over the Sobolev-type ellipsoids $ H^{\beta}(r, D) $:
When $ D = D(\delta) \to \infty  $ for $ \delta \to 0 $, the risk at $
t^{\mathfrak{b}} $ is of optimal order when $ D(\delta) $ grows faster than the
minimax truncation index $ t^{mm}_{\beta, p, r}(\delta) $, see the discussion in
Section \ref{ssec_StructuralAssumptions}. 
The same is true for the dimension-dependent error as long as $ s_{D}^{1 / ( 1 -
2 \alpha p )} \lesssim_{\alpha, A, \kappa} t^{mm}_{\beta, p, r} $. 
Therefore, we obtain:

\begin{cor}[Adaptive rates for Sobolev ellipsoids]
  \label{cor_AdaptiveUpperBoundForSobolevTypeEllipsoids}
  Assume \hyperref[eq_PSD]{\( ( \text{PSD}(p, C_{A}) ) \)} with $ \alpha p < 1 /
  2 $ and \eqref{eq_AssumptionOnKappa}. 
  Then, there exists a constant $ C_{\alpha, A, \kappa} $ depending on $ \alpha,
  p, C_{A} $ and $ C_{k} $ such that 
  \begin{align*}
    \sup_{\mu \in H^{\beta}(r, D)} 
    \mathbb{E} \| \widehat{\mu}^{( \tau_{\alpha} )} - \mu \|^{2} 
    \le C_{\alpha, A, \kappa} \mathcal{R}^{*}_{\beta, p, r}(\delta) 
  \end{align*}
  for any \( \beta, r > 0 \) with \( D \gtrsim t^{mm}_{\beta, p, r}
  \gtrsim_{\alpha, A, \kappa} s_{D}^{1 / ( 1 - 2 \alpha p )} \). 
\end{cor}

\noindent By comparing the size of $ t^{mm}_{\beta, p, r} = ( r^{2} \delta^{- 2}
)^{1 / ( 2 \beta + 2 p + 1 )} $ with $ s_{D} $, Corollary
\ref{cor_AdaptiveUpperBoundForSobolevTypeEllipsoids} yields a range of
Sobolev-type
ellipsoids $ H^{\beta}(r, D) $ for which stopping according to the smoothed
residual stopping time \( \tau_{\alpha} \) is simultaneously minimax adaptive.
How this can used to choose a suitable smoothing parameter $ \alpha $ is
further discussed in Section \ref{ssec_ChoosingTheSmoothingParameterAlpha}. 



\section{Derivation of the main results}
\label{sec_DerivationOfTheMainResults}

In this section, we derive the result in Theorem
\ref{thm_BalancedOracleInequality}. 
By defining the stochastic error term
\begin{align}
  \label{eq_StochasticError}
  S_{t}: 
  = \sum_{i = 1}^{\lfloor t \rfloor} 
    \lambda_{i}^{- 2} \delta^{2} \varepsilon_{i}^{2} 
  + ( t - \lfloor t \rfloor ) \lambda_{\lceil t \rceil}^{- 2} 
    \delta^{2} \varepsilon_{\lceil t \rceil}^{2}, 
  \qquad t \in [ 0, D ], 
\end{align}
we obtain
\begin{align}
  \mathbb{E} \| \widehat{\mu}^{( t )} - \mu \|^{2} 
  & 
  = \mathbb{E} ( B^{2}_{t}(\mu) + S_{t} ), \quad t \in [ 0, D ] 
  \\ 
  \text{and } \qquad 
  \mathbb{E} \| \hat \mu^{\tau_{\alpha}} - \mu \|^{2} 
  & 
  = \mathbb{E} ( B^{2}_{\tau_{\alpha}}(\mu) + S_{\tau_{\alpha}} ). 
\end{align}
This allows to decompose the difference between the risk at the smoothed
residual stopping time \( \tau_{\alpha} \) and the risk at any deterministic
index \( t \in [ 0, D ] \) into a bias part and a stochastic part:
\begin{align}
  \label{eq_RiskDecomposition}
    \mathbb{E} \| \hat \mu^{( \tau_{\alpha} )} - \mu \|^{2} 
  - \mathbb{E} \| \hat \mu^{( t )} - \mu \|^{2}
  \le \mathbb{E} \big( B^{2}_{\tau_{\alpha}}(\mu) - B^{2}_{t}(\mu) \big)^{+} 
    + \mathbb{E} ( S_{\tau_{\alpha}} - S_{t} )^{+}. 
\end{align}

\subsection{An oracle-proxy inequality}
\label{ssec_OracleProxyInequalities}

Initially, we compare the risk at the smoothed residual stopping time \(
\tau_{\alpha} \) with the risk at the oracle-proxy index \( t^{*}_{\alpha} \). 
For the bias part in \eqref{eq_RiskDecomposition}, we can further decompose:
\begin{align}
  \label{eq_BiasErrorDecomposition}
  \mathbb{E} \big( B^{2}_{\tau_{\alpha}}(\mu) & - B^{2}_{t}(\mu) \big)^{+} 
    \le \lambda_{\lceil t \rceil}^{- ( 2 + 2 \alpha )} 
        \mathbb{E} \big( B^{2}_{\tau_{\alpha}, \alpha}(\mu) - B^{2}_{t, \alpha}(\mu) \big)^{+} \\ 
  \notag 
  & \le \lambda_{\lceil t \rceil}^{- ( 2 + 2 \alpha )}
        \Big[ 
          \mathbb{E} \big( 
            B^{2}_{\tau_{\alpha}, \alpha}(\mu) - B^{2}_{t^{*}_{\alpha}, \alpha}(\mu)
          \big)^{+} 
        + \big( 
            B^{2}_{t^{*}_{\alpha}, \alpha}(\mu) - B^{2}_{t, \alpha}(\mu)
          \big)^{+} 
        \Big].
\end{align}
In Appendix \ref{ssec_ProofAppendixForTheMainResult}, we bound the probability
\( \mathbb{P} \{ \tau_{\alpha} \le m \} \) for \( m \ge 0 \) to derive the
following estimate for the first term in the square brackets:

\begin{prp}
  \label{prp_OracleProxyInequalityForTheAlphaBias}
  For any signal \( \mu \in \mathbb{R}^{D} \), we have
  \begin{align*}
    \mathbb{E} \big( 
      B_{\tau_{\alpha}, \alpha}^{2}(\mu) - B_{t^{*}_{\alpha}, \alpha}^{2}(\mu) 
    \big)^{+} 
    \le C ( B^{2}_{t^{*}_{\alpha}, \alpha}(\mu) + s_{D} \delta^{2} ),
  \end{align*}
  where \( C \ge 1 \) is an absolute constant.
\end{prp}

\noindent Plugging the bound from Proposition
\ref{prp_OracleProxyInequalityForTheAlphaBias} into
\eqref{eq_BiasErrorDecomposition} gives an inequality for the bias part in
\eqref{eq_RiskDecomposition}. 

\begin{cor}
  \label{cor_BiasInequalityForDeterministicIndices}
  For any signal \( \mu \in \mathbb{R}^{D} \) and \( t \in [ 0, D ] \), we have 
  \begin{align*}
    \mathbb{E} \big(
      B^{2}_{\tau_{\alpha}}(\mu) - B^{2}_{t}(\mu) 
    \big)^{+} 
    & \le C \lambda_{\lceil t \rceil}^{- ( 2 + 2 \alpha )}
          \big( 
            B^{2}_{t^{*}_{\alpha}, \alpha}(\mu) + s_{D} \delta^{2} 
          \big),
  \end{align*}
  where \( C \ge 1 \) is an absolute constant.
\end{cor}

In Appendix \ref{ssec_ProofAppendixForTheMainResult}, we also bound the
probability \( \mathbb{P} \{ \tau_{\alpha} \ge m \} \) for \( m \ge 0 \), which
yields the following bound for the stochastic part in
\eqref{eq_RiskDecomposition}:

\begin{prp}
  \label{prp_OracleProxyInequalityForTheStochasticError}
  Assume \hyperref[eq_PSD]{\( ( \text{PSD}(p, C_{A})) \)} with \( \alpha p < 1 /
  2\) and \eqref{eq_AssumptionOnKappa}.  
  Then, 
  \begin{align*}
    \mathbb{E} ( S_{\tau_{\alpha}} - S_{t^{*}_{\alpha}} )^{+} 
    \le C_{\alpha, A, \kappa} \big( 
          V_{t^{*}_{\alpha}} 
        + s_{D}^{( 2 p + 1 ) / ( 1 - 2 \alpha p )} \delta^{2} 
        \big),
  \end{align*}
  where \( C_{\alpha, A, \kappa} \ge 1 \) is a constant depending on \( \alpha,
  p, C_{A} \) and \( C_{\kappa} \).  
\end{prp}

Together, Corollary \ref{cor_BiasInequalityForDeterministicIndices} and
Proposition \ref{prp_OracleProxyInequalityForTheStochasticError} show that under
a set of fairly general assumptions, the risk at the smoothed residual stopping
time \( \tau_{\alpha} \) essentially behaves like the risk at the deterministic
oracle-proxy index \( t^{*}_{\alpha} \).
Note that the result holds for all signals \( \mu \in \mathbb{R}^{D} \) and not
only for Sobolev-type ellipsoids.

\begin{thm}[Oracle-proxy inequality]
  \label{thm_OracleProxyInequality}
  Assume \hyperref[eq_PSD]{\( ( \text{PSD}(p, C_{A}) ) \)} with \( \alpha p < 1
  / 2 \) and \eqref{eq_AssumptionOnKappa}. 
  Then, there exists a constant \( C_{\alpha, A, \kappa} \) depending on \(
  \alpha, p, C_{A} \) and \( C_{\kappa} \) such that 
  \begin{align*}
    \mathbb{E} \| \hat \mu^{( \tau_{\alpha} )} - \mu \|^{2} 
    \le C_{\alpha, A, \kappa} \big( 
          \mathbb{E} \| \hat \mu^{( t^{*}_{\alpha} )} - \mu \|^{2} 
        + s_{D}^{( 2 p + 1 ) / ( 1 - 2 \alpha p )} \delta^{2}
        \big).
  \end{align*}
  For \( t^{*}_{\alpha} \gtrsim_{\alpha, A, \kappa} s_{D}^{1 / ( 1 - 2 \alpha
  p)} \), the risk at \( \tau_{\alpha} \) is of the order of the risk at \(
  t^{*}_{\alpha} \).
\end{thm}

\begin{proof}
  After plugging the inequalities from Corollary
  \ref{cor_BiasInequalityForDeterministicIndices} and Proposition
  \ref{prp_OracleProxyInequalityForTheStochasticError} into Equation
  \eqref{eq_RiskDecomposition} with \( t = t^{*}_{\alpha} \),
  only the remaining bias part has to be estimated.
  For any \( m \ge 0 \), however, we have \( \lambda_{m}^{- ( 2 + 2 \alpha )}
  B^{2}_{m, \alpha}(\mu) \le B^{2}_{m}(\mu) \). 
  This yields
  \begin{align}
    \lambda_{\lceil t^{*}_{\alpha} \rceil}^{- ( 2 + 2 \alpha )} 
    ( B^{2}_{t^{*}_{\alpha}, \alpha}(\mu) + s_{D} \delta^{2} ) 
    & 
    \lesssim_{A} B^{2}_{t^{*}_{\alpha}}(\mu) 
               + ( t^{*}_{\alpha} )^{2 p + 2 \alpha p} s_{D} \delta^{2} 
    \\ 
    \notag 
    & 
    \lesssim_{A} B^{2}_{t^{*}_{\alpha}}(\mu) + V_{t^{*}_{\alpha}} 
               + s_{D}^{( 2 p + 1 ) / ( 1 - 2 \alpha p )} \delta^{2}
  \end{align}
  by distinguishing the cases where \( t^{*}_{\alpha} \) is smaller or greater
  than \( s_{D}^{1 / ( 1 - 2 \alpha p ) } \).
\end{proof}

The proof of Proposition \ref{prp_OracleProxyInequalityForTheStochasticError}
relies on the growth of \( 
  m \mapsto V_{m, \alpha} - B^{2}_{m, \alpha}(\mu)  
\) for \( m \ge \lceil t^{*}_{\alpha} \rceil \), which can be insufficient for
\( \alpha p \ge 1 / 2 \) even if we assume that \( \mu \in H^{\beta}(r, D) \). 
This suggests that a result as in Theorem \ref{thm_OracleProxyInequality} for \(
\alpha p \ge 1 / 2 \) requires additional assumptions on the decay of \( m
\mapsto B^{2}_{m, \alpha}(\mu) \).
We note a sufficient condition from the literature, see e.g.
\citet{KindermannNeubauer2008} or \citet{SzaboEtal2015}.

\begin{rem}[Oracle-proxy inequality under polished tails]
  \label{rem_OracleProxyInequalityUnderPolishedTails}
  Assume that the signal \( \mu \) is not only an element of \( H^{\beta}(r, D)
  \) but additionally the projection onto the first $ D $ components of an
  infinite-dimensional signal $ \tilde \mu $, which 
  satisfies a \emph{polished tail condition} of the form  
  \begin{align}
    \label{eq_PolishedTailCondition}
    \sum_{i = m}^{\infty} \tilde \mu_{i}^{2} 
    \le C_{0} \sum_{i = m}^{\rho m} \tilde \mu_{i}^{2}
    \qquad \text{ for all } m \ge 1
  \end{align}
  for an integer constant \( \rho \ge 2 \) and \( C_{0} > 0  \). 
  Then, we have \( \mathbb{E} ( S_{\tau_{\alpha}} - S_{t^{*}_{\alpha}} )^{+}
  \lesssim_{A, \kappa} ( t^{*}_{\alpha} )^{2 p + 1} \delta^{2} \) also when \(
  \alpha p \ge 1 / 2 \) and \( \kappa \ge \sum_{i = 1}^{D} \lambda_{i}^{2
  \alpha} \delta^{2} \), see Proposition
  \ref{prp_ControlOfTheStochasticErrorForAlphaPGe12}(i) in Appendix
  \ref{ssec_ProofAppendixForSupplementaryResults}.
  Under this condition, we obtain \( \mathbb{E} \| \hat \mu^{( \tau_{\alpha} )}
  - \mu \| ^{2} \lesssim_{A} \mathbb{E} \| \hat \mu^{( t^{*}_{\alpha} )} - \mu
  \|^{2} \) the same way as in Theorem
  \ref{thm_OracleProxyInequality}.  
\end{rem}


\subsection{Comparison of the oracle risks}
\label{ssec_ComparisonOfTheOracleRisks}

In this section, we derive the balanced oracle inequality in Theorem
\ref{thm_BalancedOracleInequality} from the oracle-proxy inequality in Theorem
\ref{thm_OracleProxyInequality}.
We do this by comparing the different bias and variance quantities at \(
t^{*}_{\alpha}, t^{\mathfrak{b}}_{\alpha} \) and \( t^{\mathfrak{b}} \). 
Initially, we bound the difference between the \( \alpha \)-risk terms at \(
t^{*}_{\alpha} \) and \( t^{\mathfrak{b}}_{\alpha} \). 

\begin{lem}
  \label{lem_DifferenceBetweenAlphaErrorTerms}
  We have
  \begin{align*}
    ( V_{t^{*}_{\alpha} , \alpha} - V_{t^{\mathfrak{b}}_{\alpha}, \alpha} )^{+} 
    & 
    \le \Big( 
          \sum_{i = 1}^{D} \lambda_{i}^{2 \alpha} \delta^{2} - \kappa
        \Big)^{+}
    \\ 
    \text{and } \qquad 
    \big( 
      B^{2}_{t^{*}_{\alpha}, \alpha}(\mu) 
      -
      B^{2}_{t^{\mathfrak{b}}_{\alpha}, \alpha}(\mu) 
    \big)^{+} 
    & 
    \le \Big(
          \kappa - \sum_{i = 1}^{D} \lambda_{i}^{2 \alpha} \delta^{2} 
        \Big)^{+}.
  \end{align*}
\end{lem}

\begin{proof}
  For the first inequality, we assume without loss of generality that \(
  t^{\mathfrak{b}}_{\alpha} < t^{*}_{\alpha} \).
  The monotonicity of \( t \mapsto B^{2}_{t, \alpha}(\mu) \), the fact that \(
  \mathbb{E} R^{2}_{t^{*}_{\alpha}, \alpha} = \kappa \) and Equation
  \eqref{eq_InformationInTheSmoothedResiduals} yield
  \begin{align}
    V_{t^{*}_{\alpha}, \alpha} 
    & 
    =   B^{2}_{t^{*}_{\alpha}, \alpha}(\mu) 
        + 
        \sum_{i = 1}^{D} \lambda_{i}^{2 \alpha} \delta^{2} - \kappa 
    \le B^{2}_{t^{\mathfrak{b}}_{\alpha}, \alpha}(\mu) 
        + 
        \sum_{i = 1}^{D} \lambda_{i}^{2 \alpha} \delta^{2} - \kappa 
    \\ 
    \notag 
    & 
    \le V_{t^{\mathfrak{b}}_{\alpha}, \alpha}(\mu) 
        + 
        \sum_{i = 1}^{D} \lambda_{i}^{2 \alpha} \delta^{2} - \kappa. 
  \end{align}
  
  For the second inequality, we analogously assume without loss of generality
  that \( t^{*}_{\alpha} < t^{\mathfrak{b}}_{\alpha} \).
  The monotonicity of \( t \mapsto V_{t, \alpha} \), the fact that \(
  \mathbb{E} R^{2}_{t^{*}_{\alpha}, \alpha} \le \kappa \) and Equation
  \eqref{eq_InformationInTheSmoothedResiduals} then yield
  \begin{align}
    B_{t^{*}_{\alpha} , \alpha}^{2}(\mu) 
    & 
    \le V_{t^{*}_{\alpha} , \alpha} 
        +
        \kappa - \sum_{i = 1}^{D} \lambda_{i} ^{2} \delta^{2} 
    \le V_{t^{\mathfrak{b}}_{\alpha}, \alpha} 
        +
        \kappa - \sum_{i = 1}^{D} \lambda_{i} ^{2} \delta^{2} 
    \\
    \notag 
    & 
    \le B^{2}_{t^{\mathfrak{b}}_{\alpha}, \alpha}(\mu) 
        +
        \kappa - \sum_{i = 1}^{D} \lambda_{i} ^{2} \delta^{2} .
  \end{align}
\end{proof}

\noindent Under our assumptions, the first inequality in Lemma
\ref{lem_DifferenceBetweenAlphaErrorTerms} allows to bound the size of \(
t^{*}_{\alpha} \):

\begin{cor}
  \label{cor_DistanceBetweenTjaTbaTb}
  Assume \hyperref[eq_PSD]{\( ( \text{PSD}(p, C_{A}) ) \)} with $ \alpha p < 1 /
  2 $ and \eqref{eq_AssumptionOnKappa}. 
  Then,
  \begin{align*}
    t^{*}_{\alpha} \lesssim_{\alpha, A, \kappa}
    t^{\mathfrak{b}}_{\alpha} + s_{D}^{1 / ( 1 - 2 \alpha p )} 
    \le t^{\mathfrak{b}} + s_{D}^{1 / ( 1 - 2 \alpha p )}. 
  \end{align*}
\end{cor}

\begin{proof}
  Under \hyperref[eq_PSD]{\( ( \text{PSD}(p, C_{A}) ) \)} with $ \alpha p < 1 /
  2 $, we have
  \begin{align}
    \delta^{- 2} ( V_{t^{*}_{\alpha}, \alpha} - V_{t_{\alpha}^{\mathfrak{b}}} )
    &
    \gtrsim_{A} 
      \int_{t^{\mathfrak{b}}_{\alpha}}^{t^{*}_{\alpha}} t^{- 2 \alpha p} \, d t 
    = \frac{
        ( t^{*}_{\alpha} )^{1 - 2 \alpha p} 
      - ( t^{\mathfrak{b}}_{\alpha})^{1 - 2 \alpha p}
      }{1 - 2 \alpha p}.  
  \end{align}
  Now, the result follows from Lemma \ref{lem_DifferenceBetweenAlphaErrorTerms}
  and assumption \eqref{eq_AssumptionOnKappa}. 
\end{proof}

We can now essentially compare the order of the risk at \( t^{*}_{\alpha}
\), \( t^{\mathfrak{b}}_{\alpha} \) and \( t^{\mathfrak{b}} \).

\begin{prp}[Comparison of the oracle risks]
  \label{prp_ComparisonOfTheOracleRisks}
  Assume \hyperref[eq_PSD]{\( ( \text{PSD}(p, C_{A}) ) \)} with $ \alpha p < 1 /
  2 $ and \eqref{eq_AssumptionOnKappa}. 
  Then,
  \begin{align*}
    \mathbb{E} \| \hat \mu^{( t^{\mathfrak{b}}_{\alpha} )} - \mu \|^{2}
    & \sim_{\alpha, A, \kappa} 
      \mathbb{E} \| \hat \mu^{( t^{\mathfrak{b}} )} - \mu \|^{2} \\ 
    \qquad \text{ and } \qquad 
    \mathbb{E} \| \hat \mu^{( t^{*}_{\alpha} )} - \mu \|^{2} 
    & \lesssim_{\alpha, A, \kappa} 
      \mathbb{E} \| \hat \mu^{( t^{\mathfrak{b}} )} - \mu \|^{2}
    + s_{D}^{( 2 p + 1 ) / ( 1 - 2 \alpha p )} \delta^{2}. 
  \end{align*}
\end{prp}

\begin{proof}
  For the second statement, we note that, as in \eqref{eq_RiskDecomposition}, we
  can write
  \begin{align}
    \mathbb{E} \| \hat \mu^{( t^{*}_{\alpha} )} - \mu \|^{2} 
  - \mathbb{E} \| \hat \mu^{( t^{\mathfrak{b}} )} - \mu \|^{2} 
    \le \big( 
          B^{2}_{t^{*}_{\alpha}}(\mu) - B^{2}_{t^{\mathfrak{b}}}(\mu)
        \big)^{+} 
      + ( V_{t^{*}_{\alpha}} - V_{t^{\mathfrak{b}}} )^{+}. 
  \end{align}
  We treat the two terms on the right-hand side separately. 
  For the bias part, we can assume \( t^{*}_{\alpha} \le t^{\mathfrak{b}} \). 
  Analogously to \eqref{eq_BiasErrorDecomposition}, we have
  \begin{align}
    \label{eq_ComparisonOfTheOracleRisks1}
    B^{2}_{t^{*}_{\alpha}}(\mu) - B^{2}_{t^{\mathfrak{b}}}(\mu)
    & \le \lambda_{t^{\mathfrak{b}}}^{- ( 2 + 2 \alpha )} 
          B^{2}_{t^{*}_{\alpha}, \alpha}(\mu) \\ 
    \notag 
    & \le \lambda_{t^{\mathfrak{b}}}^{- ( 2 + 2 \alpha )} 
          \big( 
            V_{t^{\mathfrak{b}}, \alpha} 
          + \kappa - \sum_{i = 1}^{D} \lambda_{i}^{2 \alpha} \delta^{2}
          \big) \\ 
    \notag 
    & \lesssim_{\alpha, A, \kappa} 
          \big( 
            ( t^{\mathfrak{b}} )^{2 p + 1} 
          + ( t^{\mathfrak{b}} )^{2 p + 2 \alpha p} s_{D} 
          \big) 
          \delta^{2},
  \end{align}
  since $ V_{t^{\mathfrak{b}}, \alpha} \lesssim_{\alpha, A, \kappa} (
  t^{\mathfrak{b}} )^{1 - 2 \alpha p} \delta^{2} $. 

  For the variance part, we can assume \( t^{*}_{\alpha} \ge t^{\mathfrak{b}} \)
  and obtain
  \begin{align}
    V_{t^{*}_{\alpha}} - V_{t^{\mathfrak{b}}} 
    & 
    \le \lambda_{\lceil t^{*}_{\alpha} \rceil}^{- ( 2 + 2 \alpha )} 
        ( V_{t^{*}_{\alpha}, \alpha} - V_{t^{\mathfrak{b}}, \alpha} ) 
    \lesssim_{A, \kappa} ( t^{*}_{\alpha} )^{2 p + 2 \alpha p} s_{D} \delta^{2}  
    \\ 
    \notag 
    & 
    \lesssim_{\alpha, A, \kappa} 
      ( 
        ( t^{\mathfrak{b}} )^{2 p + 2 \alpha p} s_{D}
      + s_{D}^{( 2 p + 1 ) / ( 1 - 2 \alpha p )}
      ) 
      \delta^{2} 
  \end{align}
  using Lemma \ref{lem_DifferenceBetweenAlphaErrorTerms} and Corollary
  \ref{cor_DistanceBetweenTjaTbaTb}.
  The intended inequality now follows from
  \begin{align}
    ( t^{\mathfrak{b}} )^{2 p + 1} \delta^{2}
    \sim_{A} V_{t^{\mathfrak{b}}} 
    \sim \mathbb{E} \| \hat \mu^{( t^{\mathfrak{b}})} - \mu \|^{2} 
  \end{align}
  and distinguishing the cases where \( t^{\mathfrak{b}} \) is smaller
  or greater than \( s_{D}^{1 / ( 1 - 2 \alpha p )} \). 

  The essential inequality "$ \lesssim_{\alpha, A, \kappa} $" in the first
  statement follows by replacing $ t^{*}_{\alpha} $ with $
  t^{\mathfrak{b}}_{\alpha} $ in \eqref{eq_ComparisonOfTheOracleRisks1} and
  noting both that $ B_{t^{\mathfrak{b}}_{\alpha}, \alpha}^{2}(\mu) =
  V_{t^{\mathfrak{b}}_{\alpha}, \alpha} $ and $ V_{t^{\mathfrak{b}}_{\alpha}}
  \le V_{t^{\mathfrak{b}}} $, since $ t^{\mathfrak{b}}_{\alpha} \le
  t^{\mathfrak{b}} $. 
  The reverse direction "$ \gtrsim_{\alpha, A, \kappa} $" follows immediately
  from the fact that the risk at $ t^{\mathfrak{b}} $ is always of smaller
  order than the risk at any other $ t \in [ 0, D ] $.
\end{proof}

\noindent Together with Theorem \ref{thm_OracleProxyInequality}, Proposition
\ref{prp_ComparisonOfTheOracleRisks} yields the result in Theorem
\ref{thm_BalancedOracleInequality}.

  

\section{Constraints in terms of lower bounds}
\label{sec_ConstraintsInTermsOfLowerBounds}

\subsection{Undersmoothing for \( \alpha p \le 1 / 4 \)}
\label{ssec_UndersmoothingForAlphaPLe14}

The first constraint in Theorem \ref{thm_BalancedOracleInequality} is the 
dimension-dependent error term
\begin{align}
  s_{D}^{( 2 p + 1 ) / ( 1 - 2 \alpha p )} \delta^{2}
  \sim_{A} V_{s_{D}^{1 / ( 1 - 2 \alpha p )}}. 
\end{align}
We show that an error of this order is unavoidable:
From the identity in  \eqref{eq_OptionalStoppingIdentity} and the monotonicity
of \( t \mapsto V_{t} \), we obtain that for any \( i_{0} \in \{ 0, \dots,
D \} \),

\begin{align}
  \label{eq_GeneralVarianceLowerBound}
  \mathbb{E} \| \widehat{\mu}^{( \tau_{\alpha} )} - \mu \|^{2} 
  \ge \mathbb{E} ( V_{i_{0}} \mathbf{1} \{ \tau_{\alpha} \ge i_{0} \} )
  \ge \mathbb{P} \{ \tau_{\alpha} \ge i_{0} \} V_{i_{0}}.
\end{align}
By considering the zero signal \( \mu = 0 \), we can isolate the error, which
stems directly from the stochastic variability of the smoothed residuals.
In Appendix \ref{ssec_ProofAppendixForSupplementaryResults}, we show that for
\( \alpha p \le 1 / 4 \), we stop later than $ i_{0} = s_{D}^{1 / ( 1 - 2 \alpha
p)} $ with nonvanishing probability for \( D = D_{\delta} \to \infty \) when \(
\delta \to 0 \).
This causes a dimension-dependent error of the size \( V_{s_{D}^{1 / ( 1 - 2
\alpha p )}} \).
Since this reasoning can be extended to \( \mu \ne 0 \), we obtain:

\begin{prp}[Dimension-dependent lower bound]
  \label{prp_DimensionDependentLowerBound}
  Assume \hyperref[eq_PSD]{\( ( \text{PSD}(p, C_{A}) ) \)} with \( \alpha p \le
  1 / 4 \) and \eqref{eq_AssumptionOnKappa}. 
  Then, we have for any \( \mu \in \mathbb{R}^{D} \) that 
  \begin{align*}
    \mathbb{E} \| \widehat{\mu}^{( \tau_{\alpha} )} - \mu \|^{2} 
    \ge C s_{D}^{( 2 p + 1 ) / ( 1 - 2 \alpha p )} \delta^{2}
  \end{align*}
  with an absolute constant \( C > 0 \), provided that \( \delta \) is
  sufficiently small and \( D = D_{\delta} \to \infty \) for \( \delta \to 0 \). 
\end{prp}

\noindent The proof of Proposition \ref{prp_DimensionDependentLowerBound} shows
that decreasing the admissible order of \( | \kappa - \sum_{i = 1}^{D}
\lambda_{i}^{2 \alpha} \delta^{2} | \) beyond \( s_{D} \delta^{2} \) does not
decrease the order of the lower bound. 
At the same time, increasing the admissible order of \( | \kappa - \sum_{i =
1}^{D} \lambda_{i}^{2 \alpha} \delta^{2} | \) may increase the order of the
lower bound.  This further motivates assumption \eqref{eq_AssumptionOnKappa}. 


\subsection{Oversmoothing for \( \alpha p \ge 1 / 2 \)}
\label{ssec_OversmoothingForAlphaPGe12}

The second constraint in Theorem \ref{thm_BalancedOracleInequality} is \( \alpha
p < 1 / 2 \).   
We already anticipated in Section \ref{ssec_MainResults} that for \( \alpha p
\ge 1 / 2 \), an oracle inequality is no longer possible, since \(
t^{\mathfrak{b}}_{\alpha} \) can be of strictly smaller order than \(
t^{\mathfrak{b}} \).
We make this precise by providing a lower bound.
Analogously to \eqref{eq_GeneralVarianceLowerBound}, the monotonicity of \( t
\mapsto B^{2}_{t}(\mu) \) yields that for any \( i_{0} \in \{ 0, \dots, D \} \),
we have
\begin{align}
  \label{eq_GeneralBiasLowerBound}
  \mathbb{E} \| \widehat{\mu}^{( \tau_{\alpha} )} - \mu \|^{2} 
  \ge \mathbb{E} ( B^{2}_{i_{0}}(\mu) \mathbf{1} \{ \tau_{\alpha} \le i_{0} \} )
  \ge \mathbb{P} \{ \tau_{\alpha} \le i_{0}  \} B^{2}_{i_{0}}(\mu).
\end{align}
Intuitively, \( \tau_{\alpha} \) centres around \( t^{*}_{\alpha} \). 
Therefore, we can hope to bound the probability in
\eqref{eq_GeneralBiasLowerBound} from below against a constant when \( i_{0} \)
is of the order of \( t^{*}_{\alpha} \). 
If \( t^{*}_{\alpha} \le t^{\mathfrak{b}}_{\alpha} \), this gives a bound in
terms of \( B^{2}_{t^{\mathfrak{b}}_{\alpha}}(\mu) \).
From  \eqref{eq_ComparisonTbaTja}, we have that \( t^{*}_{\alpha} \le
t^{\mathfrak{b}}_{\alpha} \) exactly when \( \kappa \ge \sum_{i = 1}^{D}
\lambda_{i}^{2 \alpha} \delta^{2} \).
Under this assumption, we obtain:

\begin{prp}[\( \alpha \)-balanced oracle lower bounds]
  \label{prp_AlphaBalancedOracleLowerBounds}
  Assume \hyperref[eq_PSD]{\( ( \text{PSD}(p, C_{A}) ) \)} and \( 
  \kappa \ge \sum_{i = 1}^{D} \lambda_{i}^{2 \alpha} \delta^{2} \).
  Then, there exists a constant \( C_{A}' > 0 \) depending on \( p \) and \(
  C_{A} \) such that  
  \begin{align*}
    \sup_{\mu \in H^{\beta}(r, D)} 
    \mathbb{E} \| \widehat{\mu}^{( \tau_{\alpha})} - \mu \|^{2} 
    \ge C_{A}' \mathcal{R}_{\beta, r, p, \alpha}(\delta) 
  \end{align*}
  with 
  \begin{align*}
    \mathcal{R}_{\beta, r, p, \alpha}(\delta): 
    = \begin{cases}
          r^{2} \big( r^{- 2} \delta^{2} / ( 1 - 2 \alpha p ) \big)^{2 \beta / ( 2 \beta + 2 p + 1 )},
          & \alpha p < 1 / 2, \\[5pt]
          r^{2} ( 
            r^{- 2} \delta^{2} \log(r^{2} \delta^{- 2})
          )^{2 \beta / ( 2 \beta + 2 p + 1 )}, 
          & \alpha p = 1 / 2, \\[5pt]
          r^{2} ( r^{- 2} \delta^{2} )^ {2 \beta / ( 2 \beta + 2 p + 2 \alpha p )},
          & \alpha p > 1 / 2, 
        \end{cases}
  \end{align*}
  provided that \( \delta \) is sufficiently small and \( t^{mm}_{\beta, p,
  r}(\delta) = o(D) \).
\end{prp}

\noindent The proof is postponed to Appendix
\ref{ssec_ProofAppendixForSupplementaryResults}. 

Proposition \ref{prp_AlphaBalancedOracleLowerBounds} directly reflects
the bound on \( t^{\mathfrak{b}}_{\alpha} \) from
\eqref{eq_UpperBoundForTheAlphaBalancedOracle}.
As long as \( \alpha p < 1 / 2 \), the lower bound is of the order of the
minimax rate \( \mathcal{R}^{*}_{\beta, r, p} (\delta) \), however, we lose a
power of \( 1 / ( 1 - 2 \alpha p ) \) in the constant.
This is exactly what would be expected from
the possible loss of smoothing in the size of \( t^{\mathfrak{b}}_{\alpha} \)
deduced in \eqref{eq_UpperBoundForTheAlphaBalancedOracle}. 
Note that this result also implies that the constant in Theorem
\ref{thm_BalancedOracleInequality} grows at least this fast  in \( \alpha \). 
For \( \alpha p \ge 1 / 2 \), the balanced oracles \( t^{\mathfrak{b}}_{\alpha}
\) and \( t^{\mathfrak{b}} \) are of different order.
Since \( \tau_{\alpha} \) reflects the size of \( t^{\mathfrak{b}}_{\alpha} \)
rather than \( t^{\mathfrak{b}} \), we oversmooth and stop too early such that
rate optimal adaptation is no longer possible.

For \( \alpha = 1 \) and \( p > 1 / 2 \), the lower bound for \( \alpha p > 1 /
2 \) in Proposition \ref{prp_AlphaBalancedOracleLowerBounds} is the same rate
that \citet{BlanchardMathe2012} achieve via the discrepancy principle for the
normal equation (up to a \( \log \)-factor).
In our setting, this also is the correct rate. 
In Appendix \ref{ssec_ProofAppendixForSupplementaryResults}, we separately
control the stochastic error for \( \alpha p \ge 1 / 2 \).
We can then prove:

\begin{prp}
  \label{cor_RatesForAlphaPGe12}
  Assume \hyperref[eq_PSD]{\( ( \text{PSD}(p, C_{A}) ) \)} with \( \alpha p \ge
  1 / 2\), \( \kappa \ge \sum_{i = 1}^{D}
  \lambda_{i}^{2 \alpha} \delta^{2} \)
  and \eqref{eq_AssumptionOnKappa}.
  Then, there exists a constant \( C_{A, \kappa} \) depending on \( p, C_{A} \)
  and \( C_{\kappa} \) such that 
  \begin{align}
    \sup_{\mu \in H^{\beta}(r, D)}
    \mathbb{E} \| \widehat{\mu}^{( \tau_{\alpha} )} - \mu \|^{2} 
    \le C_{A, \kappa} \mathcal{R}_{\beta, r, p, \alpha}(\delta) 
  \end{align}
  with \( \mathcal{R}_{\beta, r, p, \alpha}(\delta) \) from Proposition
  \ref{prp_AlphaBalancedOracleLowerBounds}. 
\end{prp}

\begin{proof}
  From \eqref{eq_RiskDecomposition} with \( t = t^{mm}_{\beta, p, r, \alpha} \)
  and Proposition \ref{prp_ControlOfTheStochasticErrorForAlphaPGe12}(ii) in
  Appendix \ref{ssec_ProofAppendixForSupplementaryResults}, we obtain
  \begin{align}
    & \ \ \ \ 
    \mathbb{E} \| \hat \mu^{( \tau_{\alpha} )} - \mu \|^{2} 
  - \mathbb{E} \| \hat \mu^{( t^{mm}_{\beta, p, r, \alpha} )} - \mu \|^{2} \\ 
    \notag 
    & \lesssim_{A, \kappa} 
      \lambda_{\lceil t^{mm}_{\beta, p, r, \alpha} \rceil}^{- ( 2 + 2 \alpha)} 
      \big( B^{2}_{t^{*}_{\alpha}, \alpha}(\mu) + s_{D} \delta^{2} \big) 
    + ( t^{mm}_{\beta, p, r, \alpha} )^{2 p + 1} \delta^{2} \\ 
    \notag 
    & \lesssim_{A, \kappa} 
      \lambda_{\lceil t^{mm}_{\beta, p, r, \alpha} \rceil}^{- ( 2 + 2 \alpha)} 
      \big( V_{t^{*}_{\alpha}, \alpha} + s_{D} \delta^{2} \big) 
    + ( t^{mm}_{\beta, p, r, \alpha} )^{2 p + 1} \delta^{2} \\ 
    \notag 
    & \lesssim_{A, \kappa} 
      ( t^{mm}_{\beta, p, r, \alpha} )^{2 p + 2 \alpha p} 
      V_{t^{mm}_{\beta, p, r, \alpha}, \alpha} 
    + ( t^{mm}_{\beta, p, r, \alpha} )^{2 p + 1} \delta^{2}.
  \end{align}
  Since \( V_{t^{mm}_{\beta, p, r, \alpha}} \sim_{A} \log(t^{mm}_{\beta, p, r,
  \alpha}) \delta^{2} \) for \( \alpha p = 1 / 2 \) and \( V_{t^{mm}_{\beta, p,
  r, \alpha}} \sim_{A} \delta^{2} \) for \( \alpha p > 1 / 2 \), this gives the
  result. 
\end{proof}

\noindent Note that in addition to \eqref{eq_AssumptionOnKappa}, an assumption
on \( \kappa \) such as \( \kappa \ge \sum_{i = 1}^{D} \lambda_{i}^{2 \alpha}
\delta^{2}  \) is necessary for \( \alpha p \ge 1 / 2 \). 
Otherwise, \( \kappa = 0 \) satisfies \eqref{eq_AssumptionOnKappa}, which yields
that \( \tau_{\alpha} = D \).  



\section{Discussion and simulations}
\label{sec_DiscussionAndSimulations}
  
The results from Sections \ref{sec_FrameworkForTheAnalysisAndMainResults},
\ref{sec_DerivationOfTheMainResults} and
\ref{sec_ConstraintsInTermsOfLowerBounds} reveal three different smoothing
regimes:
For $ \alpha p \le 1 / 4 $, the risk at $ t^{\mathfrak{b}}_{\alpha} $ is of the
same order as the risk at $ t^{\mathfrak{b}} $. 
There is, however, an dimension-dependent error present and we potentially stop too late
when $ t^{\mathfrak{b}} \not \gtrsim_{\alpha, A, \kappa} s_{D}^{1 / ( 1 - 2
\alpha p )} $, i.e. we undersmooth.
For $ 1 / 4 < \alpha p < 1 / 2 $, the risk at $ t^{\mathfrak{b}}_{\alpha} $ is
still of the same order as the risk at $ t^{\mathfrak{b}} $ and the
dimension-dependent
error disappears.
Note, however, that we lose in the constant $ C_{\alpha, A, \kappa} $ from
Theorem \ref{thm_BalancedOracleInequality}, which was discussed in detail after
Proposition \ref{prp_AlphaBalancedOracleLowerBounds}. 
For $ 1 / 2 \le \alpha p $, the risk at $ t^{\mathfrak{b}}_{\alpha} $ can be of
smaller order than the risk at $ t^{\mathfrak{b}} $. 
We potentially stop too early, i.e. we oversmooth.
This is summarised in Table \ref{tab_SmoothingRegimes}. 

\begin{table}[t]
  \centering
  \begin{tabular}{lll}
    \\[-1.8ex]\hline
    \hline\\[-1.8ex]
    \( \alpha p < 1 / 4 \)  & \( 1 / 4 \le \alpha p < 1 / 2 \) & \( 1 / 2 \le \alpha p \) \\
    \\[-1.8ex]\hline
    \\[-1.8ex]
    Risk at \( t^{\mathfrak{b}}_{\alpha} \sim \) Risk at \( t^{\mathfrak{b}} \)
    &
    Risk at \( t^{\mathfrak{b}}_{\alpha} \sim \) Risk at \( t^{\mathfrak{b}} \)
    & 
    Risk at \( t^{\mathfrak{b}}_{\alpha} \not \lesssim \) Risk at \( t^{\mathfrak{b}} \),
    \\
    & & \( t^{\mathfrak{b}}_{\alpha} \not \gtrsim t^{\mathfrak{b}} \), stop too early 
    \\[+1.8ex]
    Dimension error & No dimension error & Dimension error
    \\
    Stop too late for \( t^{\mathfrak{b}} \not \gtrsim s_{D}^{1 / ( 1 - 2 \alpha p )} \) & &
    \\[+1.8ex]
    Undersmoothing          & Loss in the constant             & Oversmoothing  \\
    \\[-1.8ex]\hline
    \hline\\[-1.8ex]
  \end{tabular}
  \caption{Overview of the smoothing regimes.}
  \label{tab_SmoothingRegimes} 
\end{table}

\noindent In Section \ref{ssec_EstimationResultsForSimulatedData}, we discuss
particular choices of \( \alpha \) and in Section
\ref{ssec_EstimationResultsForSimulatedData}, we compare our theoretical results
with the estimation results for simulated data.

\subsection{Choosing the smoothing parameter $ \alpha $}
\label{ssec_ChoosingTheSmoothingParameterAlpha}

We consider the problem of choosing a suitable smoothing parameter $ \alpha \ge
0$ in order to adaptively estimate signals from $ H^{\beta}(r, D) $ for fixed $
r > 0 $ and a range of smoothness levels $ \beta $ in $ [ \beta_{\min}, \infty
) $. 
Here, we assume that $ \beta_{\min} $ is a minimal a priori smoothness available to
the user. This yields the minimally sufficient approximation dimension \( D
\sim t^{mm}_{\beta_{\min}, p, r} = ( r^{2} \delta^{- 2} )^{1 / ( 2 \beta_{\min}
+ 2 p + 1 )} \), see the discussion in Section \ref{ssec_StructuralAssumptions}. 
Note that the choice $ \beta_{\min} = 0 $, which provides a sufficient
approximation for any degree of smoothness $ \beta \ge 0 $, may already be
computationally feasible. 
For $ D \sim t^{mm}_{\beta_{\min}, p, r}  $, the size of the standard deviation
term is of order
\begin{align}
  s_{D} \sim_{\alpha, A} 
  \begin{cases}
    ( t^{mm}_{\beta_{\min}, p, r} )^{1 / 2 - 2 \alpha p}, & \alpha p < 1 / 4, \\[2pt]
    \log t^{mm}_{\beta_{\min}, p, r},                     & \alpha p = 1 / 4, \\[2pt]
    1,                                                    & \alpha p > 1 / 4. 
  \end{cases}
\end{align}

When a maximal degree of smoothness $ \beta_{\max} $ is known, the
user may consider the tradeoff between the smoothing parameter \( \alpha \)
and the constant \( C_{\alpha, A, \kappa} \) in Theorem
\ref{thm_BalancedOracleInequality}. 
The optimal smoothing index is then given by the
smallest \( \alpha \), which guarantees adaptation over all \( \beta \in [
  \beta_{\min}, \beta_{\max} ] \). 
By Corollary \ref{cor_AdaptiveUpperBoundForSobolevTypeEllipsoids}, this index is
given by the smallest \( \alpha \in [ 0, 1 / ( 4 p ) ) \) such that 
\begin{align}
  t^{mm}_{\beta_{\max} , p, r} 
  & \gtrsim_{\alpha, A, \kappa} 
    ( t^{mm}_{\beta_{\min}, p, r} )^{\frac{1 / 2 - 2 \alpha p}{1 - 2 \alpha p}}, \\ 
  \notag 
  \quad \text{ i.e. } \quad    
  2 \beta_{\max} + 2 p + 1 
  & \le \frac{1 - 2 \alpha p}{1 / 2 - 2 \alpha p}
        ( 2 \beta_{\min} + 2 p + 1  ). 
\end{align}

When no such \( \beta_{\max} \) is known, the natural choice for the smoothing
index is \( \alpha = 1 / ( 4 p) \), which is the smallest index at which the
dimension dependent error is of lower order for any $ \beta \ge \beta_{\min} $:
Theorem \ref{thm_BalancedOracleInequality} together with
\eqref{eq_ComparabilityOfOracleRisks} 
yields that for all \( \beta \ge \beta_{\min} \),
\begin{align}
  \mathbb{E} \| \hat \mu^{( \tau_{1 / ( 4 p )} )} - \mu \|^{2} 
  \le C_{A, \kappa} \Big( 
        \min_{t \in [ 0, D ]} 
        \mathbb{E} \| \hat \mu^{( t )} - \mu \|^{2} 
      + ( \log t^{mm}_{\beta_{\min}, p, r} )^{2 ( 2 p + 1 )} \delta^{2} 
  \Big)
\end{align}
with a constant \( C_{A, \kappa} \) depending on \( p, C_{A} \) and \( C_{\kappa}
\). 
For any \( \beta \ge \beta_{\min} \), \( t^{mm}_{\beta, p, r} \) is
essentially larger than \( \log t^{mm}_{\beta_{\min}, p, r} \) up to a constant
depending on \( \beta \).
Therefore, for any \( \beta \ge \beta_{\min} \), 
\begin{align}
  \mathbb{E} \| \hat \mu^{( \tau_{1 / ( 4 p )} )} - \mu \|^{2} 
  \le C_{A, \kappa, \beta} \mathcal{R}^{*}_{\beta, p, r}(\delta) 
  \qquad \text{for all } \mu \in H^{\beta}(r, D)
\end{align}
with a constant \( C_{A, \kappa, \beta} > 0 \) which depends on \( p, C_{A},
C_{\kappa} \) and \( \beta \). 

This clearly shows the advantage of smoothing compared to no smoothing: We can
directly influence the range of adaptation, whereas whithout smoothing, the
range is fixed and we cannot expect to adapt to signals of smoothness greater
than $ 2 \beta_{\min} + p + 1 / 2 $. 
Additionally, the discussion above yields a natural choice for $ \alpha $, i.e. $
\alpha = 1 / ( 4 p ) $, which in particular depends only on the degree of the
polynomial spectral decay $ p $.
This choice can further be optimised given additional information about $
\beta_{\max} $. 

Finally, we may not have access to arbitrary powers of \( ( A A^{\top} ) \)
and only be able to choose between \( \alpha = 0 \) and \( \alpha = 1 \). 
For the direct comparison of nonsmoothed residual stopping and the discrepancy
principle for the normal equation, our results show the following:
As long as \( p < 1 / 2 \), we should clearly prefer the $ \alpha = 1 $. 
When \( p \) is only slightly larger than \( 1 / 2 \), no method is clearly
better than the other and our choice should depend on the size of \( D \) and
possibly additional prior knowledge about the signals we want to estimate.
Finally, when \( p \) is substantially larger than \( 1 / 2 \),
we should prefer nonsmoothed residual stopping. 
In particular, the two-step procedure from \citet{BlanchardEtal2018a} -- when
computationally affordable -- should produce uniformly better results, since we
neither pay in the rate nor in the constant.


\subsection{Estimation results for simulated data}
\label{ssec_EstimationResultsForSimulatedData}

In this section, the properties of smoothed residual stopping, which have been
analysed in the previous sections are illustrated by Monte Carlo simulations.
Analogous to the simulations in \citet{BlanchardEtal2018a}, we set
\begin{align}
  \delta = 0.01, \quad 
  p = 0.5, \quad 
  \lambda_{i} = i^{- p}, \quad i = 1, ... D 
  \qquad \text{ and } \quad 
  \kappa = \sum_{i = 1}^{D} \lambda_{i}^{2 \alpha} \delta^{2}
\end{align}
such that \( t^{*}_{\alpha} = t^{\mathfrak{b}}_{\alpha} \). 
In this setting, the natural parameter choice from Section
\ref{ssec_ChoosingTheSmoothingParameterAlpha} is $ \alpha = 1 / ( 4 p ) = 0.5 $. 
The threshhold at which we enter the oversmoothing regime is $ \alpha = 1 / (
2 p ) = 1 $.
We consider the signals \( \mu^{( \infty )}, \mu^{( 3.0 )}, \mu^{( 2.1 )} \) and
\( \mu^{( 0.5)} \) defined by
\begin{align}
  \mu^{(\infty)}_{i} & = 5 \exp(- 0.1 i), \qquad
  \qquad \ \ \ \ \,
  \mu^{(3.0)}_{i} = 500 | U_i | i^{- 2.05}, \quad 
  \\ 
  \notag 
  \mu^{(2.1)}_{i} & = 5000 | \sin(0.01 i) | i^{- 1.6}, \qquad 
  \mu^{(0.5)}_{i} = 250 | \sin(0.002 i) | i^{- 0.8},
\end{align}
with \( ( U_{i} )_{i \le D} \) independent standard uniform random variables.
\( \mu^{( \infty )}, \mu^{( 2.1 )} \) and \( \mu^{( 0.5 )} \) are the
\emph{supersmooth}, \emph{smooth} and \emph{rough} signals from
\cite{BlanchardEtal2018a}, respectively.
The random signal \( \mu^{( 3.0 )} \) will further illustrate the effect of
gradually increasing the smoothing index \( \alpha \). 
All signals are indexed by their smoothness parameter \( 2 \beta \)
for the corresponding Sobolev-type ellipsoid \( H^{\beta}(r, D) \), i.e. they
are ordered \( ( \mu^{( \infty )}, \mu^{( 3.0 )}, \mu^{( 2.1 )}, \mu^{( 0.5 )})
\) from smooth to rough.
The SVD coefficients \( ( \mu_{i} )_{i \le D} \) of the signals and their decay
are illustrated in Figure \ref{fig_FourSignals}.

\begin{figure}[t]
  \centering
  \includegraphics[width=\textwidth]{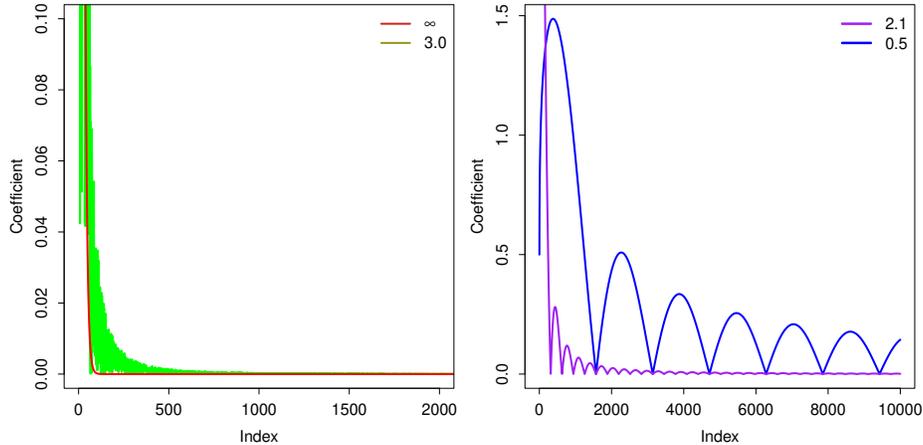}
  \caption{SVD coefficients for four signals of different smoothness.}
  \label{fig_FourSignals} 
\end{figure}

Initially, we set \( D = D_{\delta} = 10000 \) to make our results directly
comparable with \cite{BlanchardEtal2018a}. 
In this setting, the integer valued classical oracle indices of \( ( \mu^{(
\infty )}, \mu^{( 3.0 )}, \mu^{( 2.1 )}, \mu^{( 0.5 )}) \) are given by \( (43,
58, 504, 1331) \).
The balanced counterparts are \( (37, 52, 445, 2379) \). 
For any of the signals, 1000 realisations of the model
\begin{align}
  \label{eq_SimulationModel}
  Y_{i} & = \lambda_{i} \mu_{i} + \delta \varepsilon_{i}, \qquad 
  i = 1, \dots D
\end{align}
are simulated. For each of these, we calculate the smoothed residual stopping
time \( \tau_{\alpha} \) for smoothing parameters \( \alpha \in \{ 0, 0.2, 0.5, 1,
1.5 \} \).
As in \cite{BlanchardEtal2018a}, we compute the
\emph{relative efficiency} 
\begin{align}
  \big( 
    \min_{m \le D} 
    \mathbb{E} \| \widehat{\mu}^{( m )} - \mu \|^{2}
  \big)^{1 / 2} 
  / 
  \| \widehat{\mu}^{( \tau_{\alpha} )} - \mu \|,
\end{align}
which serves as an estimate for the inverse of the square root
of the constant between \( \mathbb{E} \| \widehat{\mu}^{( \tau_{\alpha})} -
\mu\|^{2} \) and \( \mathbb{E} \| \widehat{\mu}^{( t^{\mathfrak{c}} )} - \mu
\|^{2} \).
Additionally, we determine the \emph{relative stopping time} \( \lceil
t^{\mathfrak{b}}_{\alpha} \rceil / \tau_{\alpha} \).
Boxplots of these quantities are presented in Figure
\ref{fig_RelativeEfficiency}. 

The simulation of the relative efficiency closely matches the theoretical
results.
For no to little smoothing of the residuals, i.e. $ \alpha \in \{ 0, 0.2 \} $,
the risk of estimating the smooth signals $ \mu^{( \infty )} $ and $ \mu^{( 3
)} $ is clearly dominated by the dimension dependent error term in Theorem
\ref{thm_BalancedOracleInequality}, i.e. we are in the undersmoothing regime,
see Table \ref{tab_SmoothingRegimes}. 

\begin{figure}[H]
  \centering
  \includegraphics[width=\textwidth]{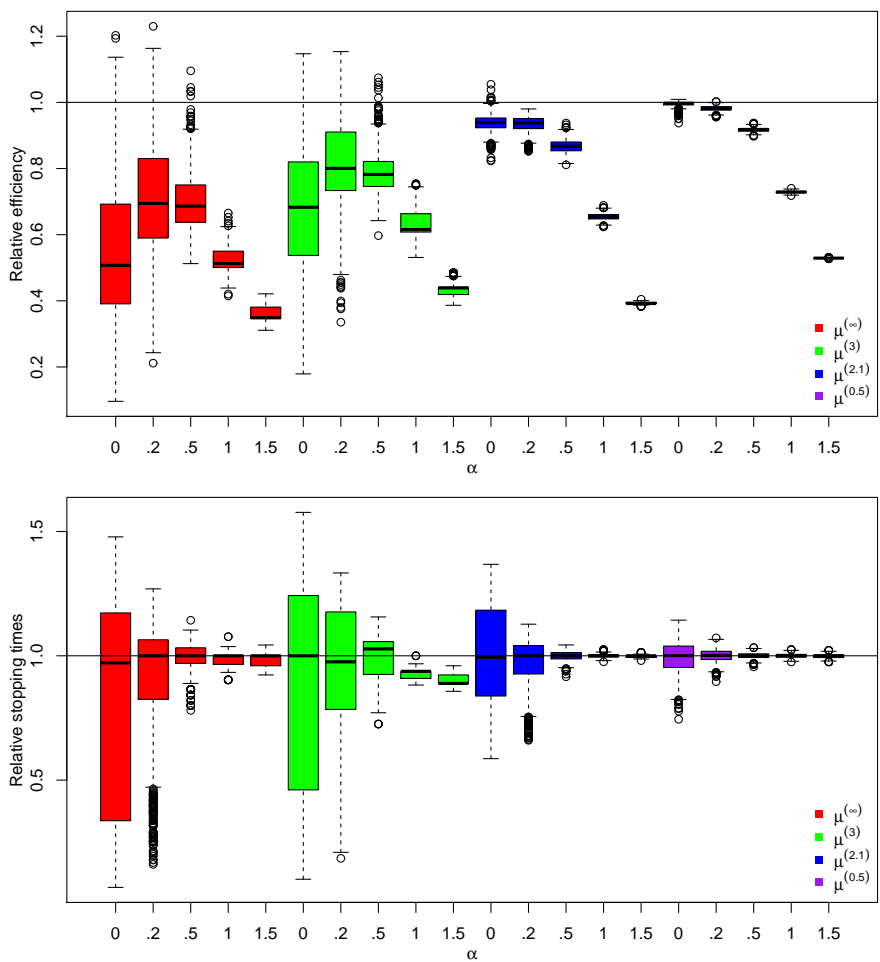}
  \caption{Boxplots of simulation results for \( D = 10000 \).}
  \label{fig_RelativeEfficiency} 
\end{figure}

\noindent This is evident, since the relative efficiency does not concentrate
well and can take values close to zero, i.e. the loss at the stopping time can
be much larger than the oracle risk. 
Smoothing is able to mitigate this.
Indeed, for the natural parameter choice $ \alpha = 1 / ( 4 p ) = 0.5 $, the
relative efficiency concentrates around a reasonable constant across all
signals.
\noindent Note, however, that for the rougher signals $ \mu^{( 2.1 )} $ and $ \mu^{(
0.5 )} $, smoothing has worsened the constant.
This shows that the tradeoff between the range of adaptation and the constant
discussed in Sections \ref{ssec_OversmoothingForAlphaPGe12} and
\ref{ssec_ChoosingTheSmoothingParameterAlpha} cannot be neglected in practice.
Finally, we observe a clear dropoff in the quality of estimation over all
signals for $ a \ge 1 / ( 2 p ) = 1 $, which is also expected from Table
\ref{tab_SmoothingRegimes}, since we are entering the oversmoothing regime.

\begin{figure}[H]
  \centering
  \includegraphics[width=\textwidth]{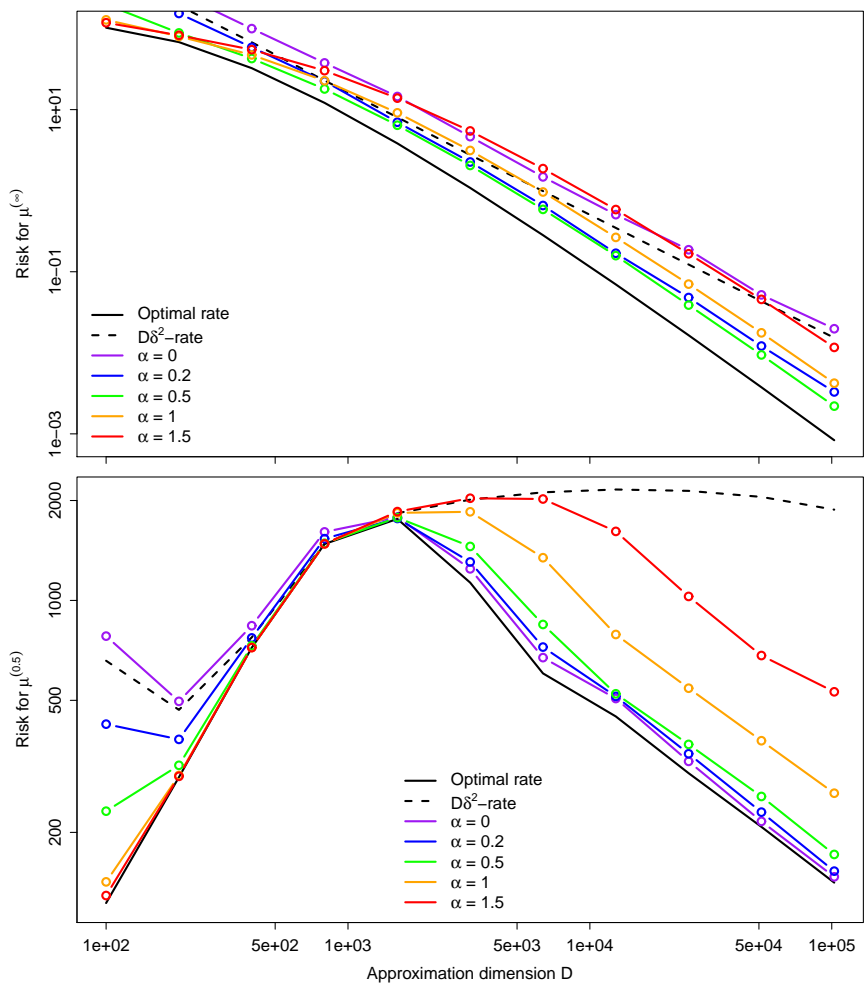}
  \caption{
    Log-log plot of the risk rates for \( \mu^{( \infty )} \) and \( \mu^{(
    0.5)} \)  and different smoothing indices.
  } 
  \label{fig_RateBehaviour} 
\end{figure}

The same effects are illustrated by the behaviour of the stopping time itself.
The boxplots of \( \lceil t^{\mathfrak{b}}_{\alpha} \rceil / \tau_{\alpha} \)
reflect our findings from Section \ref{ssec_OracleProxyInequalities} that
\( \tau_{\alpha} \) centers around \( t^{*}_{\alpha} \), which is equal to \(
t^{\mathfrak{b}}_{\alpha} \) in our case.
For \( \alpha \in \{ 0, 0.2 \} \), we are in the undersmoothing regime and
large deviations from \( t^{\mathfrak{b}}_{\alpha} \) are possible due to the
result in Proposition \ref{prp_DimensionDependentLowerBound}. 
By gradually increasing \( \alpha \), these vanish and for \( \alpha \ge 1 \),
\( \tau_{\alpha} \) evermore resembles the deterministic stopping time \(
t^{\mathfrak{b}}_{\alpha} \). 
Numerical evaluation of \( t^{\mathfrak{b}}_{\alpha} \) shows that for $ \alpha
\ge 1$, \( t^{\mathfrak{b}}_{\alpha} \) itself rapidly decreases for all signals
considered, resulting in stopping times which are substantially too early.
This increases the bias of \( \hat \mu^{( \tau_{\alpha} )} \), which explains
the loss in the relative efficiency.
The size of the loss suggests that for \( \alpha \ge 1 \), we
are indeed in the oversmoothing regime.

Finally, we directly illustrate the behaviour of convergence rates in the
asymptotical setting where \( D_{\delta} \to \infty \) for \( \delta \to 0 \).
We consider the estimation for the super-smooth signal \( \mu^{( \infty )} \)
and the rough signal \( \mu^{( 0.5 )} \).
For different smoothing indices \( \alpha \), these already display all three
possible regimes for the convergence rate.
In the simulations, we use values of \( D = D_{k} = 100 \cdot 2^{k} \) for \( k
= 0, \dots, 10 \) with corresponding noise levels
\begin{align}
  \delta_{k} 
  = \sqrt{r_{\max}^{2} / D_{k}^{2 \beta_{\min} + 2 p + 1}}, 
  \qquad k = 0, \dots, 10
\end{align}
where \( r_{\max} = 1000 \)  and \( 2 \beta_{\min} = 0.5 \).
In this scenario, \( D_{0.01} = 10000 \) as before and  \( D_{\delta_{k}} \)
grows as the minimax truncation $ t^{mm}_{0.5, p, r} $ index of the
rough signal \( \mu^{( 0.5 )} \), i.e. we assume that we want to be able to
cover signals up to at least this roughness.
Again, we simulate 1000 realisations from \eqref{eq_SimulationModel} and consider
the stopped estimator for smoothing indices \( \alpha \in \{ 0, 0.2, 0.5, 1, 1.5
\} \).
We take the mean squared loss as an estimate for the risk and compare the
convergence behaviour of the stopped estimator with the optimal rate, which is
achieved by stopping at \( t^{\mathfrak{c}} \) and the rate of stopping
deterministically at \( \sqrt{D} \), which gives the dimension-dependent rate $
D^{( 2 p + 1 ) / 2} \delta^{2} = D \delta^{2} $ from Proposition
\ref{prp_DimensionDependentLowerBound} for no smoothing.
The results are displayed in Figure \ref{fig_RateBehaviour}.

We consider the results for \( \mu^{( \infty )} \).
For $ \alpha = 0 $, we are in the undersmoothing regime and obtain the $ D
\delta^{2} $-rate, i.e. we do about as good as stopping at a deterministic index
of size $ \sqrt{D} $. 
This is exactly what we would expect from the lower bound in Proposition
\ref{prp_DimensionDependentLowerBound}. 
Smoothing of the residuals improves the rate. 
Numerical calculations show that the simulated behaviour for $ \alpha = 1 / ( 4
p ) = 0.5 $ is optimal up to a factor of 2.5.
Note, however, that for $ \alpha = 1 / ( 2 p ) = 1 $, the results already
deteriorate again, which is consistent with the fact that this is the threshhold
case from Table \ref{tab_SmoothingRegimes} at which we should lose rate
optimality. 
Finally, for $ \alpha = 1.5 $, we are deep into the oversmoothing regime and
obtain substantially suboptimal behaviour.

For \( \mu^{( 0.5 )} \), the picture is different. 
Since \( \mu^{( 0.5 )} \) is particularly rough, the risk initially increases
with the approximation dimension, simply because a larger part of the signal is
considered.
As \( t^{\mathfrak{b}} \) is always substantially greater than \( \sqrt{D} \),
we never suffer from undersmoothing due to the stochastic variability of the
residuals.
Therefore, \( \alpha = 0 \) outperforms all other indices.
As predicted by Proposition \ref{prp_AlphaBalancedOracleLowerBounds}, the
results deteriorate with increasing \( \alpha \). 
For \( \alpha \in \{ 0, 0.2, 0.5 \} \), however, they group tightly
together.
This is exactly what is expected from the theoretical results, since for
smoothing up to the natural choice $ \alpha = 1 / ( 4 p ) = 0.5 $, we should
only observe a loss in the constant but not in the rate.
For values of $ \alpha $ greater than the threshhold \( 1 / ( 2 p ) = 1 \), we
clearly observe oversmoothing.

Summarising, the simulations reiterate the theoretical results from Sections
\ref{ssec_MainResults} and \ref{sec_ConstraintsInTermsOfLowerBounds} as well
as the discussion in Section \ref{ssec_ChoosingTheSmoothingParameterAlpha}.
In particular, the parameter choice $ \alpha = 1 / ( 4 p ) $ yields reasonable
estimation results across the board for all signals.
At the same time, the tradeoff between the range of adaptation and the constant
in front of the rate is important.
Therefore, if prior information about the maximal possible smoothness is
available to the user, it should be incorporated to further optimise the choice
of $ \alpha $.



\section{Appendix}
\label{sec_Appendix}

\subsection{Proof appendix for the main result}
\label{ssec_ProofAppendixForTheMainResult}

\begin{proofof}{Proposition \ref{prp_OracleProxyInequalityForTheAlphaBias}}
  \label{prf_OracleProxyInequalityForTheAlphaBias}
  If \( B_{\tau_{\alpha}, \alpha}^{2}(\mu) > B_{t^{*}_{\alpha}, \alpha}^{2}(\mu)
  \), then we have that \( \tau_{\alpha} \le \lfloor t^{*}_{\alpha} \rfloor \).
  This yields
  \begin{align}
    \mathbb{E} \big( 
      B_{\tau_{\alpha}, \alpha}^{2}(\mu) 
    - B_{t^{*}_{\alpha}, \alpha}^{2}(\mu)
    \big)^{+} 
    & = \sum_{m = 0}^{\lfloor t^{*}_{\alpha} \rfloor - 1} 
        \lambda_{m + 1}^{2 + 2 \alpha}
        \mu_{m + 1}^{2} \mathbb{P} \{ \tau_{\alpha} \le m \} \\ 
    \notag 
    & + ( t^{*}_{\alpha} - \lfloor t^{*}_{\alpha} \rfloor ) 
        \lambda_{\lceil t^{*}_{\alpha} \rceil}^{2 + 2 \alpha} 
        \mu_{\lceil t^{*}_{\alpha} \rceil}^{2} 
        \mathbb{P} \{ \tau_{\alpha} \le \lfloor t^{*}_{\alpha} \rfloor \}.
  \end{align}
  For a fixed \( m \le \lfloor t^{*}_{\alpha} \rfloor \), we consider the event 
  \( \{ \tau_{\alpha} \le m \} = \{ R^{2}_{m, \alpha} \le \kappa \} \).
  The probability of this event can be bounded by 
  \begin{align}
    & \ \ \ \
    \mathbb{P} \{ R_{m, \alpha}^{2} \le \kappa \} 
      = \mathbb{P} \Big\{ 
          \sum_{i = m + 1}^{D} 
          \lambda_{i}^{2 \alpha} (
            \lambda_{i}^{2} \mu_{i}^{2} 
            + 2 \lambda_{i} \mu_{i} \delta \varepsilon_{i} 
            + \delta^{2} \varepsilon_{i}^{2}
          ) 
          \le \kappa 
        \Big\} \\ 
    & = \mathbb{P} \Big\{ 
          \sum_{i = m + 1}^{D} \lambda_{i}^{2 \alpha} 
          \big( 
            2 \lambda_{i} \mu_{i} \delta \varepsilon_{i} 
            + \delta^{2} ( \varepsilon_{i}^{2} - 1 )
          \big)
          \le - ( \mathbb{E} R_{m, \alpha}^{2} - \kappa ) 
        \Big\}
  \end{align}
  \begin{align}
    & \le \mathbb{P} \Big\{ 
            \sum_{i = m + 1}^{D} 
            \lambda_{i}^{2 \alpha} \delta^{2} ( \varepsilon_{i}^{2} - 1 )
            \le \frac{- ( \mathbb{E} R_{m, \alpha}^{2} - \kappa )}{2} 
          \Big\} \\ 
    \notag 
    &   + \mathbb{P} \Big\{ 
            \sum_{i = m + 1}^{D} 
            \lambda_{i}^{1 + 2 \alpha} \mu_{i} \delta \varepsilon_{i} 
            \le \frac{- ( \mathbb{E} R_{m, \alpha}^{2} - \kappa )}{4} 
          \Big\} \\ 
    \label{eq_OracleProxyInequalityForTheAlphaBias1}
    & \le \exp \bigg( 
            \frac{
              - ( \mathbb{E} R_{m, \alpha}^{2} - \kappa )^{2}
            }{16 s_{D}^{2} \delta^{4}}
          \bigg) 
        + \exp \bigg( 
            \frac{- ( \mathbb{E} R_{m, \lambda}^{2} - \kappa )^{2}}{
              32 \delta^{2} 
              \sum_{i = m + 1}^{D} \lambda_{i}^{2 + 4 \alpha} \mu_{i}^{2} 
            }
          \bigg) \\ 
    \label{eq_OracleProxyInequalityForTheAlphaBias2}
    & \le \exp \bigg( 
            \frac{
              - \big( B_{m, \alpha}^{2}(\mu) - B_{t^{*}_{\alpha}, \alpha}^{2}(\mu) \big)^{2}
            }{16 s_{D}^{2} \delta^{4}}
          \bigg) 
        + \exp \bigg( 
                 \frac{
                   - \big( B_{m, \alpha}^{2}(\mu) - B_{t^{*}_{\alpha}, \alpha}^{2}(\mu) \big)^{2}
                  }{32 \delta^{2} B_{m, \alpha}^{2}(\mu)}
          \bigg).
  \end{align}
  In order to obtain \eqref{eq_OracleProxyInequalityForTheAlphaBias1}, we use
  Lemma 1 from \citet{LaurentMassart2000} and the Gaussian tail bound \(
  \mathbb{P} \{ Z \le - t \} \le e^{- t^{2} / ( 2 \sigma^{2} )} \), \( t > 0
  \), for a random variable \( Z \) distributed according to \( N(0, \sigma
  ^{2}) \).
  Further, we use that for \( m \le \lfloor t^{*}_{\alpha} \rfloor \),  
  \begin{align}
    \mathbb{E} R_{m, \alpha}^{2} - \kappa 
    & = \mathbb{E} R_{m, \alpha}^{2} - \mathbb{E} R_{t^{*}_{\alpha}, \alpha}^{2} \\ 
    \notag 
    & = B_{m, \alpha}^{2}(\mu) - B_{t^{*}_{\alpha}, \alpha}^{2}(\mu) 
      + V_{t^{*}_{\alpha}, \alpha} - V_{m, \alpha} \\ 
    \notag 
    & \ge B_{m, \alpha}^{2}(\mu) - B_{t^{*}_{\alpha}, \alpha}^{2}(\mu) 
  \end{align}
  to obtain \eqref{eq_OracleProxyInequalityForTheAlphaBias2}. 

  We set
  \begin{align}
    F(t): 
    & = \exp \Big( \frac{- t^{2}}{ 16 s_{D}^{2} \delta^{4}} \Big) 
      + \exp \Big( 
          \frac{- t^{2}}{
            32 \delta^{2} ( B_{t^{*}_{\alpha}, \alpha}^{2}(\mu) + t )
          }
        \Big), 
    \qquad t \ge 0.
  \end{align}
  The monotonicity of \( t \mapsto B_{t, \alpha} ^{2} \) and \( F \) and a
  Riemann sum approximation yield
  \begin{align}
    & \ \ \ \
    \mathbb{E} \big( 
      B^{2}_{\tau_{\alpha}, \alpha}(\mu) - B_{t^{*}_{\alpha}, \alpha}^{2}(\mu)
    \big)^{+} \\ 
    \notag 
    & \le \sum_{m = 0}^{\lfloor t^{*}_{\alpha} \rfloor - 1} 
          \lambda_{m + 1}^{2 + 2 \alpha} \mu_{m + 1}^{2} 
          F(B_{m, \alpha}^{2}(\mu) - B_{t^{*}_{\alpha}, \alpha}^{2}(\mu)) \\ 
    \notag 
    & + ( t^{*}_{\alpha} - \lfloor t^{*}_{\alpha} \rfloor ) 
        \lambda_{\lceil t^{*}_{\alpha} \rceil}^{2 + 2 \alpha} 
        \mu_{\lceil t^{*}_{\alpha} \rceil}^{2} 
        F(
          B_{\lfloor t^{*}_{\alpha} \rfloor, \alpha}^{2}(\mu) 
          - B_{t^{*}_{\alpha}, \alpha}^{2}(\mu) 
        ) \\ 
    \notag 
    & \le \int_{B^{2}_{t^{*}_{\alpha}, \alpha}(\mu)}^{\infty} 
            F(t - B^{2}_{t^{*}_{\alpha}, \alpha}(\mu)) 
          \, d t 
      \le \int_{0}^{\infty} F(t) t \, d t \\ 
    \notag 
    & \le \frac{1}{2} \sqrt{ 2 \pi 8 s_{D}^{2} \delta^{4}} 
        + \int_{0}^{B^{2}_{t^{*}_{\alpha}, \alpha}(\mu)} 
            \exp \Big( 
              \frac{- t^{2}}{ 64 \delta^{2} B^{2}_{t^{*}_{\alpha}, \alpha}(\mu)} 
            \Big)
          \, d t 
        + \int_{B^{2}_{t^{*}_{\alpha}, \alpha}(\mu)}^{\infty} 
            \exp \Big( \frac{- t}{64 \delta^{2}} \Big)
          \, d t \\ 
    \notag 
    & \le \sqrt{ 4 \pi } s_{D} \delta^{2} 
        + \frac{1}{2} \sqrt{
            2 \pi \cdot 32 \delta^{2} B^{2}_{t^{*}_{\alpha}, \alpha}(\mu) 
          } 
        + 64 \delta^{2} 
      \le 74 s_{D} \delta^{2} + 2 B^{2}_{t^{*}_{\alpha}, \alpha}(\mu). 
  \end{align}
  For the last inequality, we use the binomial identity to obtain \( \sqrt{\pi
  \delta^{2} B ^{2}_{ t^{*}_{\alpha}, \alpha}(\mu)} \le (\pi \delta^{2} +
  B^{2}_{t^{*}_{\alpha}, \alpha}(\mu)) / 2 \) and the estimate \( \sqrt{4 \pi} +
  2 \pi \le 10 \). 
\end{proofof}

\begin{proofof}{Proposition \ref{prp_OracleProxyInequalityForTheStochasticError}}
  \label{prf_OracleProxyInequalityForTheStochasticError}
  The Cauchy-Schwarz inequality and \( \mathbb{E} \varepsilon_{m}^{4} = 3 \)
  yield
  \begin{align}
    \label{eq_OracleProxyInequalityForTheStochasticError1}
    \mathbb{E} ( S_{\tau_{\alpha}} - S_{\lceil t^{*}_{\alpha} \rceil} )^{+} 
    & = \delta^{2} \sum_{m = \lceil t^{*}_{\alpha} \rceil + 1}^{D} 
        \lambda_{m}^{- 2}
        \mathbb{E} ( \varepsilon_{m}^{2} \mathbf{1} \{ \tau_{\alpha} \ge m \} ) \\ 
    \notag 
    & \le \delta^{2} \sum_{m = \lceil t^{*}_{\alpha} \rceil + 1}^{D}
          \lambda_{m}^{- 2} \sqrt{\mathbb{E} \varepsilon_{m}^{4}} 
          \sqrt{\mathbb{P} \{ \tau_{\alpha} \ge m \}} \\ 
    \notag 
    & \le \sqrt{3} \delta^{2} \sum_{m = \lceil t^{*}_{\alpha} \rceil + 1}^{D}
          \lambda_{m}^{- 2} \sqrt{\mathbb{P} \{ \tau_{\alpha} \ge m \}}. 
  \end{align}
  The smoothed residual stopping time satisfies \( \tau_{\alpha} \ge m \)
  exactly when \( R^{2}_{m - 1, \alpha} > \kappa \).
  For \( m \ge \lceil t^{*}_{\alpha} \rceil + 1 \), the probability above can
  therefore be estimated by 
  \begin{align}
    \mathbb{P} \{ R^{2}_{m - 1, \alpha} > \kappa \} 
    & = \mathbb{P} \Big\{ 
          \sum_{i = m}^{D} \lambda_{i}^{2 \alpha} 
          ( \lambda_{i} \mu_{i} + \delta \varepsilon_{i} )^{2}
          > \kappa 
        \Big\} \\ 
    & = \mathbb{P} \Big\{ 
          \sum_{i = m}^{D} \lambda_{i}^{2 + 2 \alpha} \mu_{i}^{2} 
        + 2 \lambda_{i}^{1 + 2 \alpha} \mu_{i} \delta \varepsilon_{i} 
        + \lambda_{i}^{2 \alpha} \delta^{2} \varepsilon_{i}^{2} 
          > \kappa
        \Big\} \\ 
    & \le \mathbb{P} \Big\{ 
            \sum_{i = m}^{D} 
            \lambda_{i}^{2 \alpha} \delta^{2} ( \varepsilon_{i}^{2} - 1 )  
            > \frac{\kappa - \mathbb{E} R^{2}_{m - 1, \alpha}}{2}
          \Big\} \\ 
    \notag 
    &   + \mathbb{P} \Big\{ 
            \sum_{i = m}^{D} 
            \lambda_{i}^{1 + 2 \alpha} \mu_{i} \delta \varepsilon_{i} 
            > \frac{\kappa - \mathbb{E} R^{2}_{m - 1, \alpha}}{4} 
          \Big\}
  \end{align}
  \begin{align}
    \label{eq_OracleProxyInequalityForTheStochasticError2}
    & \le \exp \bigg( 
            \frac{- ( \kappa - \mathbb{E} R^{2}_{m - 1, \alpha} )^{2}}{
              16 \sum_{i = m}^{D} \lambda_{i}^{4 \alpha} \delta^{4} 
            + 8 \delta^{2} \lambda_{m}^{2 \alpha} 
              ( \kappa - \mathbb{E} R^{2}_{m - 1, \alpha} ) 
            }
          \bigg) \\ 
    \notag 
    &   + \exp \bigg( 
            \frac{- ( \kappa - \mathbb{E} R^{2}_{m - 1, \alpha} )^{2}}{
              32 \delta^{2} 
              \sum_{i = m}^{D} \lambda_{i}^{2 + 4 \alpha} \mu_{i}^{2} 
            }
          \bigg).
  \end{align}
  The last inequality follows again from Lemma 1 in \cite{LaurentMassart2000}
  and the Gaussian tail bound \( \mathbb{P} \{ Z \le - t \} \le e^{- t^{2} / (
  2 \sigma^{2} )} \), \( t > 0 \), for a random variable \( Z \) distributed
  according to \( N(0, \sigma ^{2}) \).

  Since \( \alpha p < 1 / 2 \), we have the following essential lower bound for
  the numerator in the exponential terms in
  \eqref{eq_OracleProxyInequalityForTheStochasticError2}:
  \begin{align}
    \label{eq_OracleProxyInequalityForTheStochasticError3}
    \kappa - \mathbb{E} R^{2}_{m - 1, \alpha} 
    & \ge \mathbb{E} R^{2}_{t^{*}_{\alpha}, \alpha} 
        - \mathbb{E} R^{2}_{m - 1, \alpha} \\ 
    \notag 
    & = B^{2}_{t^{*}_{\alpha}, \alpha}(\mu) - V_{t^{*}_{\alpha}, \alpha} 
      + V_{m - 1, \alpha} - B^{2}_{m - 1, \alpha}(\mu) \\ 
    \notag 
    & \gtrsim_{A} \sum_{i = \lceil t^{*}_{\alpha} \rceil + 1}^{m - 1}
                  i^{- 2 \alpha p} \delta^{2} 
      \ge \delta^{2} 
          \int_{\lceil t^{*}_{\alpha} \rceil + 1}^{m} t^{- 2 \alpha p} \, dt  \\ 
    \notag 
    & \ge \frac{\delta^{2}}{1 - 2 \alpha p} \big( 
            m^{1 - 2 \alpha p} - \lceil t^{*}_{\alpha} \rceil^{1 - 2 \alpha p} 
          \big). 
  \end{align}
  For the denominators, we use the upper bounds
  \begin{align}
    \label{eq_OracleProxyInequalityForTheStochasticError4}
    \sum_{i = m}^{D} \lambda_{i}^{4 \alpha} & \le s_{D}^{2},
    \\ 
    \label{eq_OracleProxyInequalityForTheStochasticError5}
    \lambda_{m}^{2 \alpha} ( \kappa - \mathbb{E} R^{2}_{m - 1, \alpha} ) 
    & \le \lambda_{m}^{2 \alpha} \Big( 
            \sum_{i = 1}^{D} \lambda_{i}^{2 \alpha} \delta^{2} 
          + C_{\kappa} s_{D} \delta^{2} 
          - \sum_{i = m}^{D} \lambda_{i}^{2 \alpha} \delta^{2} 
          \Big) \\ 
    \notag 
    & \le ( 1 + C_{\kappa} ) s_{D}^{2} \delta^{2},
  \end{align}
  and
  \begin{align}
    \label{eq_OracleProxyInequalityForTheStochasticError6}
    \sum_{i = m}^{D} \lambda_{i}^{2 + 4 \alpha} \mu_{i}^{2} 
    & \le \lambda_{m}^{2 \alpha} B^{2}_{m - 1, \alpha}(\mu) \\ 
    \notag 
    & \le \lambda_{m}^{2 \alpha} \Big( 
            V_{m, \alpha} + \kappa 
          - \sum_{i = 1}^{D} \lambda_{i}^{2 \alpha} \delta^{2} 
          \Big) \\ 
    \notag 
    & \le ( 1 + C_{\kappa} ) s_{D}^{2} \delta^{2}.
  \end{align}
  for \( m \ge \lceil t^{*}_{\alpha} \rceil + 1 \). 

  Together, this yields
  \begin{align}
    \label{eq_OracleProxyInequalityForTheStochasticError7}
    \mathbb{E} ( S_{\tau_{\alpha}} - S_{\lceil t^{*}_{\alpha} \rceil} )^{+}
    \lesssim \delta^{2}  
             \sum_{m = \lceil t^{*}_{\alpha} \rceil}^{D}
             m^{2 p} 
             \exp \bigg( 
               \frac{ 
                 - ( 
                     m^{1 - 2 \alpha p} 
                   - \lceil t^{*}_{\alpha} \rceil^{1 - 2 \alpha p}
                   )^{2}
               }{C_{A, \kappa} ( 1 - 2 \alpha p )^{2} s_{D}^{2}}
             \bigg)
  \end{align}
  for a constant \( C_{A, \kappa} > 0 \) depending on \( p, C_{A} \) and \(
  \kappa \). 
  By a Riemann sum approximation, the sum in
  \eqref{eq_OracleProxyInequalityForTheStochasticError7} can essentially be
  estimated from above by
  \begin{align}
    \label{eq_OracleProxyInequalityForTheStochasticError8}
    & \ \ \ \ 
    \int_{\lceil t^{*}_{\alpha} \rceil}^{\infty} 
      t^{2 p}
      \exp \bigg( 
        \frac{
          - ( t^{1 - 2 \alpha p} - \lceil t^{*}_{\alpha} \rceil^{1 - 2 \alpha p})^{2}
        }{C_{A, \kappa} ( 1 - 2 \alpha p )^{2} s_{D}^{2}}
      \bigg)
    \, d t \\ 
    \notag 
    & \sim_{\alpha, A} 
      \int_{\lceil t^{*}_{\alpha} \rceil^{1 - 2 \alpha p}}^{\infty} 
        u^{\frac{2 p + 2 \alpha p}{1 - 2 \alpha p}} 
        \exp \bigg( 
          \frac{
            - ( u - \lceil t^{*}_{\alpha} \rceil^{1 - 2 \alpha p} )^{2}
          }{C_{A, \kappa} ( 1 - 2 \alpha p )^{2} s_{D}^{2}}
        \bigg)
      \, d u \\ 
    \notag 
    & \lesssim_{\alpha, A}  
      \int_{0}^{\infty} 
        ( u + \lceil t^{*}_{\alpha} \rceil^{1 - 2 \alpha p} )^{
          \frac{2 p + 2 \alpha p}{1 - 2 \alpha p}
        } 
        \exp \bigg( 
          \frac{- u^{2}}{C_{A, \kappa} ( 1 - 2 \alpha p )^{2} s_{D}^{2}} 
        \bigg)
      \, d u \\ 
    \notag 
    & \lesssim_{\alpha, A, \kappa} 
      \lceil t^{*}_{\alpha} \rceil^{2 p + 2 \alpha p} s_{D} 
    + s_{D}^{( 2 p + 1 ) / ( 1 - 2 \alpha p )}.  \\ 
    \notag 
    & \lesssim_{\alpha, A, \kappa} 
      ( t^{*}_{\alpha} )^{2 p + 1}
    + s_{D}^{( 2 p + 1 ) / ( 1 - 2 \alpha p )}. 
  \end{align}
  Noting that \( \mathbb{E} ( S_{\lceil t^{*}_{\alpha} \rceil} - S
  _{t^{*}_{\alpha}} )^{+} \lesssim_{A} ( t^{*}_{\alpha} )^{2 p} \delta^{2} \)
  and \( V_{t} \sim_{A} t^{2 p + 1} \delta^{2} \) yields the result.
\end{proofof}


\subsection{Proof appendix for supplementary results}
\label{ssec_ProofAppendixForSupplementaryResults}

\begin{proofof}{Proposition \ref{prp_DimensionDependentLowerBound}}
  For \( \mu = 0 \) and a fixed \( i_{0} \), we have \( \tau_{\alpha} \ge
  i_{0} \) if and only if 
  \begin{align}
    \label{eq_DimensionDependentLowerBound1}
    R^{2}_{i_{0} - 1, \alpha} 
    = \sum_{i = i_{0}}^{D} 
      \lambda_{i}^{2 \alpha} ( \lambda_{i} \cdot 0 + \delta \varepsilon_{i} )^{2} 
    = \sum_{i = i_{0}}^{D}
      \lambda_{i}^{2 \alpha} \delta^{2} \varepsilon_{i} ^{2} 
    > \kappa.
  \end{align}
  This condition can be reformulated to
  \begin{align}
    \label{eq_DimensionDependentLowerBound2}
    \sum_{i = i_{0}}^{D} 
    \lambda_{i}^{2 \alpha} ( \varepsilon_{i}^{2} - 1 )
    - \sum_{i = 1}^{i_{0} - 1} \lambda_{i}^{2 \alpha} 
    & > \delta^{- 2} \kappa - \sum_{i = 1}^{D} \lambda_{i}^{2 \alpha}.
  \end{align}
  Assumption \eqref{eq_AssumptionOnKappa} and the fact that \( \sum_{i =
  1}^{i_{0} - 1} \lambda_{i}^{2 \alpha} \lesssim_{\alpha, A} i_{0}^{1 - 2
  \alpha p} \) imply that there exists a constant \( C_{\alpha, A, \kappa} > 0
  \) depending only on \( \alpha, p, C_{A} \) and \( C_{\kappa} \) such that
  for \( i_{0} \sim s_{D}^{1 / ( 1 - 2 \alpha p ) } \), 
  \begin{align}
    \label{eq_DimensionDependentLowerBound3}
    s_{D}^{-1} 
    \sum_{i = i_{0}}^{D} 
    \lambda_{i}^{2 \alpha} ( \varepsilon_{i}^{2} - 1 )
    > C_{\alpha, A, \kappa} 
  \end{align}
  is sufficient for \eqref{eq_DimensionDependentLowerBound2}. 
  Since \( \alpha p \le 1 / 4 \), the left-hand side normalises:
  We have
  \begin{align}
    \tilde s_{D}^{2}: 
    & = \text{Var} \sum_{i = i_{0}}^{D} 
       \lambda_{i}^{2 \alpha} ( \varepsilon_{i}^{2} - 1 ) \\ 
    \notag 
    & = 2 \sum_{i = i_{0}}^{D} \lambda_{i}^{4 \alpha}
    \gtrsim_{A} \begin{cases}
                  D^{1 - 4 \alpha p} - i_{0}^{1 - 4 \alpha p},
                  & \alpha p < 1 / 4, \\ 
                  \log D - \log i_{0},
                  & \alpha p = 1 / 4, 
                \end{cases}
  \end{align}
  which implies that \( \tilde s_{D}^{2} \to \infty \) for \( \delta \to 0 \),
  since \( i_{0} = o(D) \). 
  This yields that the sum in \eqref{eq_DimensionDependentLowerBound3} satisfies
  Lindeberg's condition.
  By Slutzky's Lemma, the left-hand side in
  \eqref{eq_DimensionDependentLowerBound3} then converges in distribution to a
  centred Gaussian random variable \( Z \) and
  \begin{align}
    \mathbb{P} \Big\{ 
      s_{D}^{-1} \sum_{i = i_{0} + 2}^{D} 
      \lambda_{i}^{2 \alpha} ( \varepsilon_{i}^{2} - 1 ) 
      > C_{\alpha, A, \kappa} 
    \Big\} 
    \xrightarrow[]{\delta \to 0} \mathbb{P} \{ Z > C_{\alpha, A, \kappa} \} > 0. 
  \end{align}
  This implies that \( \mathbb{P} \{ \tau_{\alpha} \ge i_{0} \} \ge C \) for
  some constant \( C > 0 \) and \( \delta \) sufficiently small.
  Together with \eqref{eq_GeneralVarianceLowerBound}, this gives
  \begin{align}
    \label{eq_ResultFor0}
    \mathbb{E} \| \widehat{\mu}^{( \tau_{\alpha} )} - 0 \|^{2} 
    \ge \mathbb{P} \{ \tau_{\alpha} \ge i_{0} \} V_{i_{0}} 
    \ge C s_{D}^{( 2 p + 1 ) / ( 1 - 2 \alpha p )} \delta^{2},
  \end{align}
  since \( V_{i_{0}} \sim_{A} i_{0}^{2 p + 1} \delta^{2} \). 

  Finally, we note that for \( \mu \ne 0 \), 
  \begin{align}
    \mathbb{P} \{ R^{2}_{i_{0}, \alpha} \ge \kappa \} 
    & = \mathbb{P} \Big\{ 
          \sum_{i = i_{0}}^{D}
          \lambda_{i}^{2 + 2 \alpha} \mu_{i}^{2} 
        + 2 \lambda_{i}^{1 + 2 \alpha} \mu_{i} \delta \varepsilon_{i} 
        + \lambda_{i}^{2 \alpha} \delta^{2} \varepsilon_{i}^{2} 
        > \kappa
        \Big\} \\ 
    \notag 
    & \ge \mathbb{P} \Big\{ 
            \sum_{i = i_{0}}^{D}
            2 \lambda_{i}^{1 + 2 \alpha} \mu_{i} \delta \varepsilon_{i} 
          + \lambda_{i}^{2 \alpha} \delta^{2} \varepsilon_{i}^{2} 
          > \kappa
          \Big\} \\ 
    \notag 
    & \ge \mathbb{P} \Big\{ 
            \sum_{i = i_{0}}^{D}
            2 \lambda_{i}^{1 + 2 \alpha} \mu_{i} \delta \varepsilon_{i} 
            \ge 0, 
            \sum_{i = i_{0}}^{D}
            \lambda_{i}^{2 \alpha} \delta^{2} \varepsilon_{i}^{2} 
          > \kappa
          \Big\} \\ 
    \notag 
    & = \frac{1}{2}
        \mathbb{P} \Big\{ 
          \sum_{i = i_{0}}^{D}
          \lambda_{i}^{2 \alpha} \delta^{2} \varepsilon_{i}^{2} 
          \ge \kappa
        \Big\},
  \end{align}
  which shows that \eqref{eq_ResultFor0} also holds for \( \mu \ne 0 \). 
\end{proofof}

\begin{proofof}{Proposition \ref{prp_AlphaBalancedOracleLowerBounds}}
  \label{prf_AlphaBalancedOracleLowerBounds}
  We consider a signal \( \mu = \mu(\delta) \in H^{\beta}(r, D) \) with
  only one nonzero coefficient at position \( i_{0} + 1 \) given by
  \begin{align}
    \label{eq_AlphaBalancedOracleLowerBounds1}
    \mu_{i_{0} + 1}^{2} 
    & : = \lambda_{i_{0} + 1}^{- (2 + 2 \alpha)} 
          \sum_{i = 2}^{i_{0}} \lambda_{i}^{2 \alpha} \delta^{2} 
    \quad \text{and} \quad 
    \mu_{i} : = 0 
    \ \ \text{ for all } i \ne i_{0} + 1. 
  \end{align}
  Note that the coefficient \( \mu_{i_{0} + 1} \) is chosen in a way that the \(
  \alpha \)-balanced oracle \( t^{\mathfrak{b}}_{\alpha} \) is slightly smaller
  than \( i_{0} \) but of the same order.
  Under the assumption on \( \kappa \), a sufficient condition for the stopping
  criterion \( R_{i_{0}, \alpha}^{2} \le \kappa \) is given by
  \begin{align}
    \label{eq_AlphaBalancedOracleLowerBounds2}
    \sum_{i = 2}^{i_{0}} \lambda_{i}^{2 \alpha} \delta^{2} 
    + 
    2 \lambda_{i_{0} + 1}^{1 + 2 \alpha} \mu_{i_{0} + 1}
    \delta \varepsilon_{i_{0} + 1} 
    + 
    \sum_{i = i_{0} + 1}^{D} 
    \lambda_{i}^{2 \alpha} \delta^{2} \varepsilon_{i}^{2} 
    & \le \sum_{i = 1}^{D} \lambda_{i}^{2 \alpha} \delta^{2}.
  \end{align}
  We consider the different regimes of \( \alpha p \):
  \begin{enumerate}[label=(\alph*)]
    \item 
      If \( \alpha p \le 1 / 4 \), then we consider the condition
      \begin{align}
        \label{eq_AlphaBalancedOracleLowerBounds3}
        \varepsilon_{i_{0} + 1} \in [ - 1, 0 ] 
        \qquad \text{ and } \qquad 
        \sum_{i = i_{0} + 2}^{D} 
        \lambda_{i}^{2 \alpha} ( \varepsilon_{i}^{2} - 1 )
        & \le 0, 
      \end{align}
      which is sufficient for \eqref{eq_AlphaBalancedOracleLowerBounds2}.
      Due to the independence of the \( ( \varepsilon_{i} )_{i \le D }
      \), we only have to control the second part of the event defined by
      \eqref{eq_AlphaBalancedOracleLowerBounds3}.
      If we choose \( i_{0} = i_{0}(\delta) \lesssim t^{mm}_{\beta, p, r}(\delta)
      \), then the standardisation of this term
      normalises in the same way as in the proof of Proposition
      \ref{prp_DimensionDependentLowerBound} due to the growth condition on \(
      D \). 
      We have 
      \begin{align}
        \lambda_{i_{0} + 1}^{- ( 2 + 2 \alpha )} 
        \sum_{i = 2}^{i_{0}} \lambda_{i}^{2 \alpha}
        \sim_{A}  \frac{1}{1 - 2 \alpha p} ( i_{0} + 1 )^{2 p + 1}
      \end{align}
      for \( i_{0} \) sufficiently large.
      Therefore, we can choose 
      \begin{align}
        i_{0} 
        \sim ( ( 1 - 2 \alpha p ) \delta^{- 2} r^{2} )^
             {1 / ( 2 \beta + 2 p + 1 )}
      \end{align}
      when \( \delta \) is sufficiently small while still maintaining \( \mu \in
      H^{\beta}(r, D) \).
      This yields
      \begin{align}
        \mathbb{E} \| \widehat{\mu}^{( \tau_{\alpha} )} - \mu \|^{2} 
        & \gtrsim_{A} ( 1 - 2 \alpha p )^{-1} i_{0}^{2 p + 1} \delta^{2} \\ 
        \notag 
        & \gtrsim_{A}
          r^{2} ( r^{- 2} \delta^{2} / ( 1 - 2 \alpha p ) )^{
                  2 \beta / ( 2 \beta + 2 p + 1 ) 
                }.
      \end{align}

    \item If \( 1 / 4 < \alpha p < 1 / 2  \), then we rearrange
      \eqref{eq_AlphaBalancedOracleLowerBounds2} to 
      \begin{align}
        \label{eq_AlphaBalancedOracleLowerBounds4}
        2 \lambda_{i_{0} + 1}^{\alpha} 
        \sqrt{\sum_{i = 2}^{i_{0}} \lambda_{i}^{2 \alpha}}
        \varepsilon_{i_{0} + 1} 
        + 
        \sum_{i = i_{0} + 1}^{D} 
        \lambda_{i}^{2 \alpha} ( \varepsilon_{i}^{2} - 1 )
        & \le \lambda_{1}^{2 \alpha}. 
      \end{align}
      If we choose \(
        i_{0} = i_{0}(\delta) \lesssim t^{mm}_{\beta, p, r}(\delta)
      \), both terms on the left-hand side of
      \eqref{eq_AlphaBalancedOracleLowerBounds4} converge to zero in
      probability, since their variances are multiples of 
      \begin{align}
        \lambda_{i_{0} + 1}^{2 \alpha} 
        \sum_{i = 2}^{i_{0}} \lambda_{i}^{2 \alpha}
        & \lesssim_{A} ( i_{0} + 1 )^{1 - 4 \alpha p} 
        \ \ \text{and} \ \
        \sum_{i = i_{0} + 1}^{D} 
        \lambda_{i}^{4 \alpha} 
        \lesssim_{A} \sum_{i = i_{0} + 1}^{D} i^{- 4 \alpha p},
      \end{align}
      which both vanish for \( \delta \to 0 \).
      Since \( \lambda_{1}^{2 \alpha} > 0 \), this yields \( \mathbb{P} \{
        \tau_{\alpha} \le i_{0} \} \to 1 \) for \( \delta \to 0 \), which gives
      the same result as in (a).

    \item If \( \alpha p \ge 1 / 2 \), the same reasoning as in (b) allows to
      bound the probability \( \mathbb{P} \{ \tau_{\alpha} \le i_{0} \} \) from
      below for \( \delta \to 0 \). 
      Since
      \begin{align}
        \lambda_{i_{0} + 1}^{- ( 2 + 2 \alpha )} 
        \sum_{i = 2}^{i_{0}} 
        \lambda_{i}^{2 \alpha} 
        \sim_{A} \begin{cases}
                   i_{0}^{2 p + 1} \log(i_{0}), & \alpha p = 1 / 2, \\ 
                   i_{0}^{2 p + 2 \alpha p},    & \alpha p > 1 / 2,  
                 \end{cases}
      \end{align}
      we can choose \( i_{0} \) of order \( t^{mm}_{\beta, p, r, \alpha}(\delta)
      \) while still maintaining \( \mu \in H^{\beta}(r, D) \).
      This yields the bound
      \begin{align}
        \mathbb{E} \| \widehat{\mu}^{( \tau_{\alpha} )} - \mu  \|^{2} 
        \gtrsim_{A} \begin{cases}
                      ( t^{mm}_{\beta, p, r, \alpha} )^{2 p + 1 }
                      \log(t^{mm}_{\beta, p, r, \alpha}) \delta^{2},
                      & \alpha p = 1 / 2, \\[5pt]
                      r^{2} ( r^{- 2} \delta^{2} )^{
                        2 \beta / ( 2 \beta + 2 p + 2 \alpha p  )
                      },
                      & \alpha p > 1 / 2.  
                    \end{cases}
      \end{align}
      This finishes the result.
  \end{enumerate}
\end{proofof}

\begin{prp}[Control of the stochastic error for \( \alpha p \ge 1 / 2 \)]
  \label{prp_ControlOfTheStochasticErrorForAlphaPGe12}
  Assume \hyperref[eq_PSD]{\( ( \text{PSD}(p, C_{A}) ) \)} with \( \alpha p \ge
  1 / 2 \), \( \kappa \ge \sum_{i = 1}^{D} \lambda_{i}^{2
  \alpha} \delta^{2} \) and \eqref{eq_AssumptionOnKappa}. 
  Then, we have the following control over the stochastic error:
  \begin{enumerate}[label=(\roman*)]
    \item For any \( \mu \in H^{\beta}(r, D) \) which is the $ D $-dimensional
      projection of a signal satisfying the polished tail condition
      \eqref{eq_PolishedTailCondition}, there exists a constant \( C_{A, \kappa}
      > 0 \) depending on \( p, C_{A} \) and \( C_{\kappa} \) such that 
      \begin{align*}
        \mathbb{E} ( S_{\tau_{\alpha}} - S_{t^{*}_{\alpha}} )^{+} 
        \le C_{A, \kappa} ( t^{*}_{\alpha} )^{2 p + 1} \delta^{2}.
      \end{align*}
    \item For any \( \mu \in H^{\beta}(r, D) \), there exists a constant \(
      C_{A, \kappa} > 0 \) depending on \( p, C_{A} \) and \( C_{\kappa} \)
      such that 
      \begin{align*}
        \mathbb{E} ( S_{\tau_{\alpha}} - S_{t^{mm}_{\beta, p, r, \alpha}} )^{+} 
        \le C_{A, \kappa} ( t^{mm}_{\beta, p, r, \alpha} )^{2 p + 1} \delta^{2}.
      \end{align*}
  \end{enumerate}
\end{prp}

\begin{proof}
  We proceed as in the proof of Proposition
  \ref{prp_OracleProxyInequalityForTheStochasticError} up to the inequality 
  in \eqref{eq_OracleProxyInequalityForTheStochasticError2}. 
  We split the two exponential terms in three and estimate from above with 
  \begin{align}
    \label{eq_ControlOfTheStochasticErrorForAlphaPGe121}
    \exp \bigg( 
      \frac{
        - ( \kappa - \mathbb{E} R^{2}_{m - 1, \alpha} )^{2}
      }{
        32 \sum_{i = m}^{D} \lambda_{i}^{4 \alpha} \delta^{4} 
      }
    \bigg) 
    & 
    + \exp \bigg( 
      \frac{
        - ( \kappa - \mathbb{E} R^{2}_{m - 1, \alpha} )
      }{
        16 \delta^{2} \lambda_{m}^{2 \alpha} 
      }
    \bigg) 
    \\ 
    \notag 
    & 
    + \exp \bigg( 
      \frac{
        - ( \kappa - \mathbb{E} R^{2}_{m - 1, \alpha} )^{2}
      }{
        32 \delta^{2} \sum_{i = m}^{D} 
        \lambda_{i}^{2 + 4 \alpha} \mu_{i}^{2} 
      }
    \bigg). 
  \end{align}

  For (i), we have
  \begin{align}
    \label{eq_ControlOfTheStochasticErrorForAlphaPGe122}
    B^{2}_{\lceil t^{*}_{\alpha} \rceil, \alpha}(\mu) 
    & = \sum_{i = \lceil t^{*}_{\alpha} \rceil + 1}^{D} 
        \lambda_{i}^{2 + 2 \alpha} \mu_{i}^{2} \\ 
    \notag
    & \lesssim_{A} 
      \lceil t^{*}_{\alpha} \rceil^{- ( 2 p + 2 \alpha p )}  
      \sum_{i = \lceil t^{*}_{\alpha} \rceil + 1}
          ^{\rho ( \lceil t^{*}_{\alpha} \rceil + 1 )} 
      \mu_{i}^{2} \\ 
    \notag
    & \lesssim_{A}  
      \sum_{i = \lceil t^{*}_{\alpha} \rceil + 1}
          ^{\rho ( \lceil t^{*}_{\alpha} \rceil + 1 )} 
      \lambda_{i}^{2 + 2 \alpha} \mu_{i}^{2}.
  \end{align}
  Choosing \( m - 1 \ge \rho ( \lceil t^{*}_{\alpha} \rceil + 1 ) \), we obtain
  that
  \begin{align}
    B^{2}_{m - 1, \alpha}(\mu) 
    = \sum_{i = m}^{D} \lambda_{i}^{2 + 2 \alpha} \mu_{i}^{2} 
    \le c_{A} B^{2}_{\lceil t^{*}_{\alpha} \rceil, \alpha}(\mu)
  \end{align}
  for a constant \( c_{A} < 1 \) depending on \( p, C_{A} \). 

  For (ii), choosing \( m - 1 \ge C_{A}' t^{mm}_{\beta, p, r, \alpha} \) for a
  constant \( C'_{A} > 1 \) depending on \( p \) and \( C_{A} \) yields
  \begin{align}
    \label{eq_ControlOfTheStochasticErrorForAlphaPGe123}
    B^{2}_{m - 1, \alpha}(\mu) 
    \lesssim_{A} 
    r^{2} ( C_{A}' t^{mm}_{\beta, p, r, \alpha} )^{- ( 2 \beta + 2 p + 2 \alpha p )}. 
  \end{align}

  Setting \( \bar t: = t^{*}_{\alpha} \) for (i) or \( \bar t: = t^{mm}_{\beta,
  p, r, \alpha} \) for (ii), we can therefore choose a constant \( C_{A}' > 1
  \) such that for \( m - 1 \ge C_{A}' \lceil \bar t \rceil \), 
  \begin{align}
    \kappa - \mathbb{E} R^{2}_{m - 1, \alpha} 
    & 
    = \kappa - \sum_{i = 1}^{D} \lambda_{i}^{2 \alpha} \delta^{2} 
    + V_{m - 1, \alpha} - B_{m - 1, \alpha}(\mu) 
    \\ 
    \notag 
    & 
    \ge \begin{cases}
          \kappa - \sum_{i = 1}^{D} \lambda_{i}^{2 \alpha} \delta^{2} 
        + V_{m - 1, \alpha} - c_{A} B_{t^{*}_{\alpha}, \alpha}(\mu), 
          & \bar t = t^{*}_{\alpha}, \\ 
          \kappa - \sum_{i = 1}^{D} \lambda_{i}^{2 \alpha} \delta^{2} 
        + V_{m - 1, \alpha} - B_{m - 1, \alpha}^{2}(\mu) 
          & \bar t = t^{mm}_{\beta, p, r, \alpha}, 
        \end{cases}
    \\ 
    \notag 
    & 
    \ge \begin{cases}
          ( 1 - c_{A} ) 
          \Big( 
            \kappa - \sum_{i = 1}^{D} \lambda_{i}^{2 \alpha} \delta^{2} 
          + V_{t^{*}_{\alpha}, \alpha}
          \Big), 
          & \bar t = t^{*}_{\alpha}, \\ 
          \kappa - \sum_{i = 1}^{D} \lambda_{i}^{2 \alpha} \delta^{2} 
        + V_{m - 1, \alpha} - B_{m - 1, \alpha}^{2}(\mu) 
          & \bar t = t^{mm}_{\beta, p, r, \alpha}, 
        \end{cases}
    \\ 
    \notag 
    & 
    \gtrsim_{A} \delta^{2},
  \end{align}
  where we have used \eqref{eq_ControlOfTheStochasticErrorForAlphaPGe123} and
  the definition of \( t^{mm}_{\beta, p, r, \alpha} \) from
  \eqref{eq_AlphaMinimaxTruncationIndex} for the last inequality. 
  Additionally, we have the estimates
  \begin{align}
    \sum_{i = m}^{D} \lambda_{i}^{4 \alpha} 
    &
    \lesssim_{A} \lambda_{m}^{\alpha}
    \\ 
    \qquad \text{ and } \qquad 
    \sum_{i = m}^{D} \lambda_{i}^{2 + 4 \alpha} \mu_{i}^{2}
    & \le \lambda_{m}^{2 \alpha} B^{2}_{m - 1, \alpha} (\mu) \\ 
    \notag 
    & \le \lambda_{m}^{2 \alpha} \Big( 
            \kappa - \sum_{i = 1}^{D} \lambda_{i}^{2 \alpha} \delta^{2} 
          + V_{m - 1, \alpha} 
          \Big) \\ 
    \notag 
    & \lesssim_{A, \kappa} \lambda_{m}^{2 \alpha} \log(m) \delta^{2},
  \end{align} 
  where we have used Equation \eqref{eq_InformationInTheSmoothedResiduals},
  assumption \eqref{eq_AssumptionOnKappa} and that without loss of generality,
  \( m \ge t^{*}_{\alpha} \). 
  Note that the log factor occurs only for \( \alpha p = 1 / 2 \). 
  We therefore obtain that for a constant \( C_{A, \kappa}'' > 0 \) depending on \(
  p \), \( C_{A} \) and \( C_{\kappa} \),
  \begin{align}
    \notag 
    \mathbb{E} ( S_{\tau_{\alpha}} - S_{\lceil \bar t \rceil} )^{+} 
    & \lesssim_{A}
      \delta^{2} 
      \sum_{m = \lceil \bar t \rceil}^{ \lceil C_{A}' \bar t \rceil} m^{2 p} 
    + \delta^{2} 
      \sum_{m = \lceil C_{A}' \bar t \rceil + 1}^{D}
      m^{2 p} \exp ( - m^{\alpha p} / ( C''_{A, \kappa} \log m ) ) \\ 
    & \lesssim_{A, \kappa} \bar t^{2 p + 1} \delta^{2}. 
  \end{align}
  Noting that \( \mathbb{E} ( S_{\lceil \bar t \rceil} - S_{\bar t} )^{+}
  \lesssim_{A} \bar t^{2 p} \delta^{2} \) finishes the proof.
\end{proof}



\newpage

\bibliographystyle{plainnat}
\bibliography{references.bib}

\end{document}